\documentclass[aap,press]{apt} \volno{58} \partno{1}
\paperno{} 
\firstpage{1} 
\usepackage{lineno,comment,amsmath,amssymb,amsfonts} 
\usepackage[colorlinks = true,linkcolor = blue,urlcolor  = blue,citecolor = blue,anchorcolor =blue]{hyperref}
\usepackage{fix-cm}

\authornames{Sarkar, Khare and Srivastava} 
\shorttitle{Moments of Gaussian condition numbers and their limits} 

\DeclareMathOperator{\Gcal}{\mathcal{G}}

\DeclareMathOperator{\Ical}{\mathcal{I}}

\DeclareMathOperator{\Ncal}{\mathcal{N}}

\DeclareMathOperator{\EE}{\mathbb{E}} 

\newcommand{\normr}[1]{\left\lVert#1\right\rVert}
\newcommand{\norm}[1]{\left\lVert#1\right\rVert_2}

\recd{}{}

\begin{document}

\title{Moment bounds for condition numbers and singular values of high-dimensional Gaussian random matrices: Applications and limitations} 

\authorone[Department of Statistics, Florida State University]{Partha Sarkar}
\authortwo[Department of Statistics, University of Florida]{Kshitij Khare} 
\authorthree[Department of Statistics and Actuarial Science, University of Iowa]{Sanvesh Srivastava}


\addressone{117 N Woodward Ave, Tallahassee, FL, USA} 
\emailone{ps24v@fsu.edu} 
\addresstwo{102 Griffin-Floyd Hall, Gainesville, FL, USA} 
\emailtwo{kdkhare@stat.ufl.edu} 
\addressthree{20 East Washington Street, Iowa City, IA, USA}
\emailthree{ssrivastva@uiowa.edu}

\begin{abstract}
Spectral properties of Gram matrices are central to high‑dimensional asymptotic analyses of statistical estimators in regression and covariance estimation. These properties, in turn, depend critically on the extreme singular values and condition numbers of Gaussian random matrices. For many applications, sharp positive and negative moment bounds for these quantities are required to control expected prediction risk and related performance metrics. Although extensive work provides concentration and tail bounds for extreme singular values of Gaussian random matrices, these results do not readily yield the moment bounds needed in such analyses. Motivated by this gap, we establish non-asymptotic moment bounds for arbitrary positive moments of the largest singular value and arbitrary negative moments of the smallest singular value, and uniform bounds for arbitrary positive moments of the condition number, of high-dimensional Gaussian random matrices. We demonstrate the utility of these bounds by applying them to derive explicit risk guarantees in high‑dimensional regression and covariance estimation, as well as to obtain bounds on the mean iteration complexity of gradient descent for solving Gram linear systems. Finally, we present counterexamples demonstrating that the positive condition‑number moment bounds and negative smallest‑singular‑value moment bounds cannot, in general, be extended to the broader class of sub‑Gaussian random matrices. 
\end{abstract}


\newpage

\section{Introduction}

\noindent Given a data matrix $X_n \in \mathbb{R}^{n \times p}$, the associated Gram matrix is defined as $S_n := \frac{1}{n} X_n^\top X_n$~\footnote{It is more common to refer to $nS_n = X_n^\top X_n$ as the Gram matrix, but we refer to $S_n$ as the Gram matrix for expositional convenience}. Gram matrices arise naturally in regression, covariance estimation, and related statistical models, where $X_n$ serves as either a design matrix or an observation matrix. Asymptotic investigations of the corresponding classical statistical estimators typically assume that the rows of $X_n$ are independently sampled from a multivariate Gaussian distribution. Modern data applications frequently operate in regimes where the dimension $p$ is comparable to, or exceeds, the sample size $n$. Consequently, meaningful asymptotic investigations require frameworks in which $p$ grows with $n$. In such \emph{high-dimensional settings}, statistical accuracy and computational stability depend critically on the spectral properties of $X_n$. In these settings, characterizing the expected prediction risk—an essential measure of estimation accuracy—frequently hinges on obtaining sharp moment bounds for the condition number, or for the extreme singular values of $X_n$ (see Section~\ref{sec:gram_app} for concrete examples). These considerations motivate the present study. We investigate arbitrary positive moments of the condition number and largest singular value of $X_n$, together with arbitrary negative moments of its smallest singular value, under regimes where $p/n \to \gamma \in [0,\infty)$.

The condition number is a central object in numerical linear algebra, governing stability, sensitivity, and algorithmic convergence. Its probabilistic behavior has therefore received sustained attention in random matrix theory. Following the seminal work of Edelman \cite{alan_edelman}, numerous studies have established several distributional properties and tail bounds for condition numbers of random matrices; see, for example, \cite{edelman, Sankar, manriquemirón2023conditionnumberrandomtridiagonal, pan2012conditionnumbersrandomtoeplitz}. For matrices with independent identically distributed (i.i.d.) Gaussian entries, a particularly rich theory is available \cite{edelman, alan_edelman, Dongarra, PhysRevE.90.050103}. However, to the best of our knowledge, existing results largely focus on tail probabilities or expectations of logarithmic functionals and do not directly yield conclusions regarding the behavior of general moments of the condition number.

As our first contribution, we investigate this question within the framework of a random design matrix $X_n \in \mathbb{R}^{n \times p}$ whose rows are i.i.d.\ multivariate Gaussian. We impose standard bounded-eigenvalue assumptions on the covariance matrix of the underlying Gaussian distribution. For any fixed positive $r$, our results (see Theorem~\ref{theorem_condition_number} and Lemma~\ref{lem:double_descent}) show that when $p/n \to \gamma \in [0,\infty) \setminus \{1\}$ as $n \to \infty$, the $r^{\text{th}}$ moment of the condition number of $X_n$ remains uniformly bounded by a constant depending only on $r$ (and not on $n$). Moreover, Lemma~\ref{lem:double_descent} demonstrates that the expected condition number diverges when $\gamma = 1$, revealing a phase-transition phenomenon consistent with---and extending---the double descent behavior anticipated in \cite{poggio2020doubledescentconditionnumber}. Although our analysis incorporates tail bounds derived in \cite{Dongarra}, the proof requires a sequence of additional, nontrivial arguments beyond those available in the existing literature.

Uniform boundedness of the expected condition number alone is not sufficient for asymptotic evaluations in many statistical applications. In particular, the study and control of negative moments of the smallest singular value of a Gaussian random matrix plays a fundamental role in problems such as Gram matrix estimation (see, e.g., \cite{ridge1,ridge2,Hastie,littlewood_offword,RudelsonVershynin2009SmallestSingularRectangular}). A classical result of \cite{RudelsonVershynin2009SmallestSingularRectangular} shows that for $n \times p$ matrices with independent sub-Gaussian entries, the smallest singular value is typically of the order of the rectangular spectral gap, namely $\sqrt{n}-\sqrt{p}$ when $n>p$, with high probability. More precisely, the associated \emph{small-ball probability} bound consists of two components: an exponentially decaying term and an anti-concentration (small-ball) term reflecting the local mass behavior of the entry distribution. While this formulation applies broadly to sub-Gaussian ensembles, it is not sufficiently sharp in the Gaussian case for establishing bounds on negative moments, due to the presence of the exponentially small term (see Section~\ref{connection:small_ball}). To address this limitation, we derive explicit non-asymptotic bounds for negative moments of the smallest singular value of a Gaussian random matrix. Using these bounds, we show that for any fixed positive $r$, the $r^{\text{th}}$ moment of $\frac{\sqrt{n}}{s_{\min}(X_n)}$ (where $s_{\min}$ denotes the smallest singular value) remains uniformly bounded by a constant depending only on $r$ (and not on $n$) under the regime $p/n \to \gamma \in [0,\infty)\setminus \{1\}$. Our analysis is based on a small-ball probability bound tailored to Gaussian matrices $X_n$ (see Theorem~\ref{min_singular_value}). Additionally, using a concentration inequality from \cite{vershynin2011introductionnonasymptoticanalysisrandom}, we establish a complementary and technically more direct result for positive moments of the largest singular value (see Lemma~\ref{max_singular_value}) under the same asymptotic regime $p/n \to \gamma \in [0,\infty)$.

Section~\ref{sec:gram_app} presents several applications of the preceding results in statistical settings involving Gram matrices. In particular, in Section~\ref{sec:ridge} we derive an explicit upper bound for the mean prediction risk (under Frobenius norm) of a ridge estimator in a general multi-response linear regression model with $p$ predictors and $q$ responses (Lemma~\ref{lem:multi_variance_term}). As a consequence, Theorem~\ref{thm:ridge_pred_bound_hd_multi} establishes that this prediction risk converges to zero under a random design framework and the standard scaling condition $pq = o(n)$. As a second application, in Section~\ref{sec:cov} we consider the sample variance--covariance matrix, a canonical example of a Gram-type matrix. Theorem~\ref{thm:inv_cov_error_rate} shows that the moment bounds for the estimation error of its inverse exhibit the same non-asymptotic rate as those of the sample covariance matrix (\cite{gram}) itself under a Gaussian design. Finally, Section~\ref {sec:optim} investigates the computational implications of our moment bounds through Gram-type linear systems. We analyze the expected iteration complexity of gradient descent for solving systems involving the sample Gram matrix. Theorem~\ref{thm:gd_gram_condition_number} establishes that, under Gaussian designs, the expected iteration complexity remains uniformly bounded whenever $p/n \to \gamma \in [0,\infty) \setminus \{1\}$. Moreover, we show that the expected iteration complexity is fundamentally governed by moments of the condition number, thereby providing a direct link between spectral stability and computational performance. In contrast, as $p/n \to 1$, the expected complexity diverges (Lemma \ref{lem:GD_worstcase}), revealing a phase-transition phenomenon consistent with the double descent behavior of the condition number.

In Section~\ref{sec_subg}, we discuss possible extensions of our results beyond the Gaussian design and clarify the inherent limitations. While many high-dimensional bounds derived under Gaussianity extend to sub-Gaussian designs via standard tail arguments, the behavior of the inverse of the smallest singular value is considerably more delicate. In particular, although the moment bounds for the largest singular value in Lemma~\ref{max_singular_value} extend to general sub-Gaussian settings (see Remark~\ref{remark:subg}), we show that the inverse-moment and condition-number results of Theorem~\ref{min_singular_value} and Theorem~\ref{theorem_condition_number} do \emph{not} hold in full generality for sub-Gaussian designs. Specifically, we construct an explicit \emph{continuous sub-Gaussian} counterexample (see Proposition~\ref{prop:subg_counterexample_r1}) for which all negative moments of the smallest singular value and all positive moments of the condition number of the corresponding design matrix are infinite. We further show that, under this construction, the associated inverse sample covariance estimation error also diverges. We relate these phenomena to the well-known role of \emph{small-ball probabilities} governing the lower tail behavior of the smallest singular value \cite{RudelsonVershynin2009SmallestSingularRectangular}, and explain why Gaussianity permits substantially sharper inverse-moment control than what is available under general sub-Gaussian small-ball theory.

The remainder of the paper is organized as follows. Section~\ref{sec:notation} introduces the notation and basic definitions used throughout the paper. The moment bounds for the smallest singular value and condition number (Theorems \ref{min_singular_value} and \ref{theorem_condition_number}) are developed in Section~\ref{sec:main_results}. Statistical implications of these bounds are examined in Section~\ref{sec:gram_app}, including applications to high-dimensional regression (Theorem \ref{thm:ridge_pred_bound_hd_multi}), covariance estimation (Theorem \ref{thm:inv_cov_error_rate}), and computational complexity (Theorem \ref{thm:gd_gram_condition_number}). Extensions beyond the Gaussian design, together with a discussion of fundamental limitations and counterexamples, are provided in Section~\ref{sec_subg}. Proofs of some of the auxiliary results are deferred to the supplementary material.

\subsection{Notation and definitions}\label{sec:notation}

\noindent Let us introduce some notation and definitions. For positive sequences $a_n$ and $b_n$,
we write $a_n = O(b_n)$ if there exists a constant $C$ such that
$a_n \leq C b_n$ for all $n \in \mathbb{N}$, and we write $a_n = \Omega(b_n)$ if there
exists a constant $C$ such that $a_n \geq C b_n$ for all $n \in \mathbb{N}$. We use $a_n = o(b_n)$ to denote that $\lim_{n \to \infty} \frac{a_n}{b_n} = 0$. Also, $a_n \sim b_n$ means that $\frac{a_n}{b_n} \to 1$ as $n \to \infty$.

Let $(X,d)$ be a metric space and let $\varepsilon>0$. A subset $N_\varepsilon\subseteq X$
is called an \emph{$\varepsilon$-net} of $X$ if for every $x\in X$ there exists
$y\in N_\varepsilon$ such that $d(x,y)\le \varepsilon$. The minimal cardinality of an
$\varepsilon$-net of $X$, if finite, is denoted by $N(X,\varepsilon)$ and is called the
\emph{covering number} of $X$.

In the following notation, $I_p$ denotes the identity matrix of order $p$. For any matrix $A\in\mathbb{R}^{n\times p}$, let $s_{\max}(A)$ denote its largest
singular value, and $s_{\min}(A)$ denote its $\min(n,p)^{th}$ singular value (when singular values are arranged in descending order). When $\mathrm{rank}(A) = \min(n,p)$, then $s_{\min}(A)$ corresponds to the smallest non-zero singular value of $A$. The condition number of $A$ is defined as $\kappa(A):=\frac{s_{\max}(A)}{s_{\min}(A)}$. 
If $A$ is a symmetric square matrix, then $\lambda_{\min}(A)$ and
$\lambda_{\max}(A)$ denote the smallest and largest eigenvalues of $A$. Also, we define the \emph{range} (or \emph{column space}) of $A$ as
\[
\operatorname{range}(A)
\;:=\;
\{\, Au : u \in \mathbb{R}^p \,\}
\;\subseteq\;
\mathbb{R}^n.
\]
Equivalently, $\operatorname{range}(A)$ is the linear span of the columns of $A$.

The unit Euclidean sphere in $\mathbb{R}^p$ is denoted by $\mathcal{S}^{p-1}$. For a
vector $x \in \mathbb{R}^p$, we denote its $r$-th norm by
$\normr{x}_r = \left(\sum_{j=1}^p \lvert x_j\rvert^r\right)^{1/r}$, and $\norm{x}$
denotes the Euclidean norm. For a $p\times p$ matrix $A=(A_{ij})_{1\leq i,j\leq p}$,
the spectral norm is defined as
$$
\norm{A}:=\underset{{u\in\mathcal{S}^{p-1}}}{\sup}\norm{A u}.
$$
We define the vectorization of the matrix $A$ as
$\mbox{vec}(A) = (A_{11},\dots,A_{p1},A_{12},\dots,A_{pp})^\top$.

\section{Main results}\label{sec:main_results} 


\noindent In this section, we derive several new results establishing upper bounds for the moments of the condition number, as well as the negative moments of the minimum singular value, of a Gaussian random matrix. These results are derived in a growing-dimension asymptotic regime under a general covariance structure. Specifically, we consider a random matrix $\tilde X_n \in \mathbb{R}^{n \times p}$ whose rows are independent and follow a multivariate Gaussian distribution with covariance matrix $\tilde\Sigma_n$. The sequence of matrices $\{\tilde\Sigma_n\}_{n \geq 1}$ are assumed to satisfy 
\begin{align}\label{cond-const}
    c_m \leq \underset{n \geq 1}{\inf} \, \lambda_{\min}(\tilde \Sigma_n) \leq \underset{n \geq 1}{\sup} \, \lambda_{\max}(\tilde \Sigma_n) \leq c_M
\end{align}
for some finite universal constants $c_m$ and $c_M$. The regularity condition in \eqref{cond-const} ensures that the sequence $\{\tilde\Sigma_n\}_{n \ge 1}$ remains uniformly well-conditioned \cite{bickel2008regularized}, in the sense that its eigenvalues are bounded away from both $0$ and $\infty$ uniformly over $n$. Assumptions of this type are standard in high-dimensional asymptotic analysis; see, for example, \cite{bickel2008regularized, spectrum, xiang, sarkar}. \textcolor{black}{Note that, since the rows of $\tilde X_n$ are assumed to be independent Gaussian, under the assumption \eqref{cond-const} we have $\operatorname{rank}(\tilde X_n) = \min(n,p)$ almost surely. Hence, $s_{\min}(\tilde X_n)$ denotes the smallest non-zero singular value of $\tilde X_n$ in this context.} Under this framework, we first establish Theorem~\ref{theorem_condition_number}, which demonstrates that the moments of the condition number $\kappa(\tilde X_n)$ remain uniformly bounded over $n$, provided that $p/n \to \gamma \in [0,1)$.

\begin{thm}[Moments of condition number of Gaussian random matrices]\label{theorem_condition_number}
Let $\kappa(\tilde X_n) = {s_{\max}(\tilde X_n)}/{s_{\min}(\tilde X_n)}$, where $\tilde X_n$ is an $n \times p$ random matrix with $p/n \to \gamma \in [0,1)$ as $n \to \infty$. Assume that the rows of $\tilde X_n$ are independent and identically distributed as $\Ncal(0, \tilde \Sigma_n)$, where $\tilde \Sigma_n$ satisfies the condition in (\ref{cond-const}). Then for any fixed $r\geq 0$ there exists a positive universal constant $K_1(r)$ (depends only on $r$, $\gamma$, $c_m$ and $c_M$) such that
\begin{align*}
    \EE [\{\kappa(\tilde X_n) \}^r] \leq K_1(r),
\end{align*}
for all $n$.
\end{thm}

\noindent Before providing the formal proof, we present a high-level overview of the proof strategy. As noted in the introduction, \cite{Dongarra} established an asymptotic tail bound for the condition number $\kappa(Z_n)$ of an $n \times p$ Gaussian random matrix $Z_n$ with i.i.d.\ $\mathcal{N}(0,1)$ entries. In particular, for $t > p$, they showed that the tail probability $\mathbb{P}(\kappa(Z_n) > t)$ exhibits polynomial decay of order $t^{-(n-p+1)}$. Although this result does not directly imply bounds of the form stated in Theorem~\ref{theorem_condition_number}, it provides a crucial starting point. When combined with additional nontrivial probabilistic arguments, it can be leveraged to derive the required moment bounds for the condition number. This is formalized in the proof below.

\begin{proof}
Write $\tilde X_n = Z_n \tilde \Sigma_n^{1/2}$, where $Z_n$ is an $n\times p$
random matrix with i.i.d.\ $\mathcal{N}(0,1)$ entries. Using standard singular
value inequalities,
\[
s_{\max}(AB)\le s_{\max}(A)s_{\max}(B),
\qquad
s_{\min}(AB)\ge s_{\min}(A)s_{\min}(B),
\]
and noting that
$s_{\max}(\tilde \Sigma_n^{1/2})=\lambda_{\max}(\tilde \Sigma_n)^{1/2}$ and
$s_{\min}(\tilde \Sigma_n^{1/2})=\lambda_{\min}(\tilde \Sigma_n)^{1/2}$, we obtain
\[
s_{\max}(\tilde X_n)
\le s_{\max}(Z_n)\,\lambda_{\max}(\tilde \Sigma_n)^{1/2},
\qquad
s_{\min}(\tilde X_n)
\ge s_{\min}(Z_n)\,\lambda_{\min}(\tilde \Sigma_n)^{1/2}.
\]
Consequently,
\[
\kappa(\tilde X_n)
=\frac{s_{\max}(\tilde X_n)}{s_{\min}(\tilde X_n)}
\le
\kappa(Z_n)\left(\frac{\lambda_{\max}(\tilde \Sigma_n)}{\lambda_{\min}(\tilde \Sigma_n)}\right)^{1/2}
\le
\kappa(Z_n)\left(\frac{c_M}{c_m}\right)^{1/2},
\]
where the last inequality follows from \eqref{cond-const}. Raising both sides to
the power $r$ and taking expectations yields
\begin{align}\label{thm_3.1_1}
   \mathbb{E}\big[\{\kappa(\tilde X_n)\}^r\big]
\le \left(\frac{c_M}{c_m}\right)^{r/2}\mathbb{E}\big[\{\kappa(Z_n)\}^r\big]. 
\end{align}

\noindent Therefore, it is enough to bound $\mathbb{E}\big[\{\kappa(Z_n)\}^r\big]$. Since $\kappa(Z_n)\ge 1$, the tail integration identity yields
\begin{equation}\label{eq:tail_int_kappa}
\mathbb{E}\big[\{\kappa(Z_n)\}^r\big]
=
\int_{0}^{\infty} r t^{r-1}\,\mathbb{P}\big(\kappa(Z_n)>t\big)\,dt.
\end{equation}

\noindent We next bound the two probability terms in the following inequalities separately. Define
\[
b_n := \sqrt n-\sqrt p,
\qquad
\sigma_n := \frac{b_n}{2}.
\]
Then for any $t\ge 1$, 
\[
\{\kappa(Z_n)>t\}
=
\left\{\frac{s_{\max}(Z_n)}{s_{\min}(Z_n)}>t\right\}
\subseteq
\{s_{\min}(Z_n)\le \sigma_n\}\ \cup\ \{s_{\max}(Z_n)>t\sigma_n\},
\]
and hence
\begin{equation}\label{eq:kappa_tail_union}
\mathbb{P}(\kappa(Z_n)>t)
\le
\mathbb{P}\big(s_{\min}(Z_n)\le \sigma_n\big)
+
\mathbb{P}\big(s_{\max}(Z_n)>t\sigma_n\big).
\end{equation}
For all sufficiently large $n$ (so that $n>p$ and $\sigma_n>0$),
\[
\mathbb{P}\big(s_{\min}(Z_n)\le \sigma_n\big)
=
\mathbb{P}\!\left(s_{\min}(Z_n)\le b_n-\frac{b_n}{2}\right)
\le
\exp\!\left(-c_1 \frac{b_n^2}{4}\right)
\le \exp(-c_2 n),
\]
for some positive constants $c_1$ and $c_2$ (depending only on $\gamma$), where the last two inequalities follow from \cite[Corollary 5.35]{vershynin2011introductionnonasymptoticanalysisrandom} and 
$b_n^2=(\sqrt n-\sqrt p)^2 \sim (1-\sqrt{\gamma})^2 n$ under $p/n\to\gamma\in[0,1)$.

Let $a_n:=\sqrt n+\sqrt p$. By the standard upper tail bound for $s_{\max}(Z_n)$ ( see \cite[Corollary 5.35]{vershynin2011introductionnonasymptoticanalysisrandom}), there exist positive constants $c_3$ and $c_4$ (depending only on $\gamma$) and a constant $t_\star=t_\star(\gamma)\ge 1$
such that, for all sufficiently large $n$ and all $t\ge t_\star$, it holds that
\[
\mathbb{P}\big(s_{\max}(Z_n)>t\sigma_n\big)
=
\mathbb{P}\big(s_{\max}(Z_n)>a_n+(t\sigma_n-a_n)\big)
\le
\exp\!\big(-c_3(t\sigma_n-a_n)_+^2\big) \le \exp(-c_4 n t^2),
\]
where $(t\sigma_n-a_n)_+=\max\{(t\sigma_n-a_n),0\}$, and the last inequality follows from \cite[Corollary 5.35]{vershynin2011introductionnonasymptoticanalysisrandom} together with 
$a_n^2=(\sqrt n+\sqrt p)^2 \sim (1+ \sqrt{\gamma})^2 n$ under $p/n\to\gamma\in[0,1)$.
Combining the above bounds yields that, for all $t\ge t_\star$ and all sufficiently large $n$,
\begin{equation}\label{eq:kappa_tail_final}
\mathbb{P}(\kappa(Z_n)>t)
\le
\exp(-c_2 n)+\exp(-c_4 n t^2).
\end{equation}

\noindent We now use \eqref{eq:kappa_tail_final} to bound \eqref{eq:tail_int_kappa}. We split the integral in \eqref{eq:tail_int_kappa} into three parts $I_1,I_2$, and $I_3$ as defined below, and show separately that each term is uniformly bounded for large $n$. The final result will follow by enlarging the constants to cover the remaining finite number of indices of $n$.

\begin{align}
\mathbb{E}\big[\{\kappa(Z_n)\}^r\big]
&=
\underbrace{\int_{0}^{t_\star} r t^{r-1}\,\mathbb{P}(\kappa(Z_n)>t)\,dt}_{=:I_1}
\quad+
\underbrace{\int_{t_\star}^{n} r t^{r-1}\,\mathbb{P}(\kappa(Z_n)>t)\,dt}_{=:I_2}\nonumber\\
&
+
\underbrace{\int_{n}^{\infty} r t^{r-1}\,\mathbb{P}(\kappa(Z_n)>t)\,dt}_{=:I_3}.
\label{eq:split_I1I2I3}
\end{align}

\smallskip
\noindent\emph{Bound for $I_1$.}
Since $\mathbb{P}(\kappa(Z_n)>t)\le 1$,
\[
I_1
\le
\int_0^{t_\star} r t^{r-1}\,dt
=
t_\star^r<\infty.
\]

\smallskip
\noindent\emph{Bound for $I_2$.}
Using \eqref{eq:kappa_tail_final},
\begin{align*}
I_2
&\le
\int_{t_\star}^{n} r t^{r-1}\Big(\exp(-c_2 n)+\exp(-c_4 n t^2)\Big)\,dt \\
&\le
n^r \exp(-c_2 n)
+
\int_{0}^{\infty} r t^{r-1}\exp(-c_4 n t^2)\,dt\\
&=
n^r \exp(-c_2 n) + \frac{\Gamma(r/2)}{2(c_4 n)^{r/2}}.
\end{align*}
Since $n^r\exp(-c_2 n)\to 0$, it follows that $I_2 \to 0$ as $n \to \infty$ for any fixed $r>0$.

\smallskip
\noindent\emph{Bound for $I_3$.}
Note that for $t \ge n$ and a constant $ C \in [5.013, 6.414]$, it follows from \cite[Theorem 4.5]{Dongarra} and for sufficiently large $n$ that
\begin{align}
P(\kappa(Z_n) \ge t) \le \frac{1}{\sqrt{2\pi}} \left( \frac{ Cp}{7t(n - p + 1)} \right)^{n - p + 1}. \label{eq:tail_bound}
\end{align}
Hence, it follows that
\begin{align}
I_3 &\le \int_{n}^{\infty} r t^{r-1} P(\kappa(Z_n) > t) dt \nonumber \\
&\le \frac{r}{\sqrt{2\pi}} \left( \frac{ Cp}{7(n - p + 1)} \right)^{n - p + 1} \int_{n}^{\infty} \frac{dt}{t^{(n - p - r) + 1}} \\
&= \frac{r}{\sqrt{2\pi}} \left( \frac{ Cp}{7(n - p + 1)} \right)^{n - p + 1} n^{-(n - p - r)}.
\end{align}

\noindent Hence $I_3$ converges to $0$ as $n \rightarrow \infty$, since $\frac{p}{n} \rightarrow \gamma \in [0, 1)$ and $C/7 \in (0,1)$ where $r > 0$ is fixed.
The proof follows by combining the bounds obtained for $I_1$, $I_2$, and $I_3$ together with \eqref{thm_3.1_1}. In particular, for sufficiently large $n$, these bounds imply the existence of a constant $K_1(r)$, depending only on $r$, $\gamma$, $c_m$, and $c_M$, such that
\begin{align*}
    \EE\!\left[\{\kappa(\tilde X_n)\}^r\right] \leq K_1(r).
\end{align*}
The argument is completed by, if necessary, enlarging the constant to account for the finitely many indices $n$ below this threshold.

\end{proof}

\begin{remark}[Expected logarithm of the condition number]
For rectangular Gaussian matrices, \cite{Dongarra} also provides an explicit bound on the expected logarithm of the condition number. Specifically, for $n,p\geq 2$, if $Z_n$ is an $n \times p$ matrix having i.i.d.\ $\mathcal{N}(0,1)$ entries, then their result implies
\[
\mathbb{E}\big[\log(\kappa(Z_n))\big]
\le
\log\!\left(\frac{n}{n-p+1}\right)+2.258.
\]
On the other hand, our proof establishes that $\mathbb{E}[\{\kappa(Z_n)\}^r]$ is uniformly
bounded for any fixed $r>0$ under $p/n\to\gamma\in[0,1)$. In particular, taking $r=1$ and
applying Jensen's inequality (since $\log(\cdot)$ is concave), we obtain
\[
\mathbb{E}\big[\log(\kappa(Z_n))\big]\le \log\!\Big(\mathbb{E}\big[\kappa(Z_n)\big]\Big)<\infty,
\]
for every $n$. Thus, while \cite{Dongarra} provides an explicit non-asymptotic upper bound for
$\mathbb{E}\big[\log(\kappa(Z_n))\big]$, our moment bound yields a complementary
justification of the finiteness of $\mathbb{E}\big[\log(\kappa(Z_n))\big]$ in the same
regime $p/n\to\gamma\in[0,1)$.
\end{remark}

\noindent Next, we extend Theorem \ref{theorem_condition_number} to the case $p/n \to \gamma \in [1,\infty)$. In particular, when $\gamma=1$, the smallest singular value of $\widetilde X_n$ becomes arbitrarily small, which leads to a sharp increase in $\kappa(\widetilde X_n)$. Consequently, in the critical regime $\gamma=1$, one has $\kappa(\widetilde X_n)$ diverges to $\infty$, and as the next result shows, even $\mathbb{E}\!\left[\{\kappa(\widetilde X_n)\}^r\right]$ diverges to $\infty$ for any $r > 0$. This reveals an ``expectation blow-up'' of the condition number at $\gamma=1$. In contrast, when $\gamma \in (1,\infty)$, the next result establishes uniform boundedness of the moments of $\kappa(\widetilde X_n)$ analogous to Theorem \ref{theorem_condition_number}. This shows that for large $n$ the expectation of the condition number remains
uniformly bounded when $\gamma\in[0,1)$, but as soon as $\gamma=1$ it spikes to
$\infty$, and then remains bounded again for $\gamma\in(1,\infty)$. This
phenomenon can be viewed as a manifestation of the well-known \emph{double
descent property} \cite{poggio2020doubledescentconditionnumber} in the expected
condition number. The next lemma formally establishes this property.

\begin{lem}[Double descent property of moments of condition number]\label{lem:double_descent}
Suppose $\tilde X_n$ is an $n\times p$ random matrix whose rows are i.i.d.\ $\mathcal{N}(0,\tilde\Sigma_n)$,
and $\tilde\Sigma_n$ satisfies \eqref{cond-const}. Then the following hold.
\begin{enumerate}[(i)]
\item (\emph{Wide-matrix regime}) If $p/n\to\gamma\in(1,\infty)$ as $n\to\infty$,
then for any fixed $r\ge 0$ there exists a constant $K_1^{\mathrm{row}}(r)>0$ such
that
\[
\mathbb{E}\big[\{\kappa(\tilde X_n)\}^r\big]\le K_1^{\mathrm{row}}(r),
\]
for all $n$.

\item (\emph{Critical regime}) If $p/n\to 1$ as $n\to\infty$, then for every fixed
$r>0$,
\[
\mathbb{E}\big[\{\kappa(\tilde X_n)\}^r\big]\to\infty\qquad\text{as }n\to\infty.
\]
\end{enumerate}
\end{lem}

\begin{proof}
First recall that \eqref{cond-const} implies $\tilde\Sigma_n\succ 0$ for all $n$,
and hence $\mathrm{rank}(\tilde X_n)=\min(n,p)$ almost surely.

\begin{enumerate}[(i)]
    \item When $p/n\to\gamma\in(1,\infty)$, we have $n/p\to \gamma^{-1}\in(0,1)$. Since the singular values of $\tilde X_n$ and $\tilde X_n^\top$ coincide, it follows that $\kappa(\tilde X_n)=\kappa(\tilde X_n^\top)$. Therefore, applying Theorem~\ref{theorem_condition_number} to the transposed matrix $\tilde X_n^\top$ yields the stated uniform bound.

   \item  Consider the critical regime $p/n\to 1$. By \eqref{cond-const} we may write
$\tilde X_n=Z_n\tilde\Sigma_n^{1/2}$ with $Z_n$ having i.i.d.\ $\mathcal{N}(0,1)$
entries. Moreover, by \eqref{cond-const} there exists a constant $C$ (depends only on $c_m$ and $c_M$) such that
\[
\kappa(\tilde X_n)\ge C\,\kappa(Z_n)\qquad\text{for all }n.
\]
By the Bai--Yin law \cite{bai_yin} in the regime $p/n\to 1$,
\[
\frac{s_{\max}(Z_n)}{\sqrt n}\to 2,
\qquad\text{and}\qquad
\frac{s_{\min}(Z_n)}{\sqrt n}\to 0,
\qquad\text{a.s.}
\]
hence $\kappa(Z_n)\to\infty$ almost surely, and therefore
$\kappa(\tilde X_n)\to\infty$ almost surely. Thus $\kappa(\tilde X_n)^r\to\infty$
a.s.\ for any fixed $r>0$. For every $M>0$,
$\min\{\kappa(\tilde X_n)^r,M\}\to M$ a.s.\ and
$0\le \min\{\kappa(\tilde X_n)^r,M\}\le M$, so by dominated convergence theorem,
\[
\lim_{n\to\infty}\mathbb{E}\!\left[\min\{\kappa(\tilde X_n)^r,M\}\right]=M.
\]
Since $\min\{\kappa(\tilde X_n)^r,M\}\le \kappa(\tilde X_n)^r$, we obtain
$\liminf_{n\to\infty}\mathbb{E}[\kappa(\tilde X_n)^r]\ge M$ for all $M>0$, and
therefore $\mathbb{E}[\kappa(\tilde X_n)^r]\to\infty$.

\end{enumerate}
\end{proof}

\noindent
Uniform boundedness of the condition number moments alone is not sufficient for asymptotic evaluations in many statistical applications (see Section~\ref{sec:gram_app}). In addition to the uniform boundedness of the condition number established in Theorem~\ref{theorem_condition_number}, such evaluations repeatedly rely on more refined spectral controls. Specifically, the developments in Section~\ref{sec:gram_app} require: (i) bounds on the negative moments of the smallest singular value, which govern the stability of inverses and lower spectral regularity, and (ii) bounds on the positive moments of the largest singular value, which control the magnitude of the spectral norm. 

We begin with the lower end of the singular-value spectrum. 
Theorem~\ref{min_singular_value} controls negative moments of $s_{\min}
(X_n)$ when $X_n$ has i.i.d.\ $\Ncal(0,1)$ entries. The bound has two parts: a leading term of order $(n-p-1)^{-r/2}$ and an additional term that decays exponentially fast in $(n-p-r+1)$. The leading term captures the dependence of the negative moments of the minimum singular value on the gap between the sample size and the dimension, while the exponentially small remainder reflects that extremely small values of $s_{\min}(X_n)$ become rapidly less likely as $n - p$ grows. In particular, under the growing dimension regime i.e. when $p/n\to\gamma\in[0,1)$, the lemma implies a uniform bound on $\EE\big[(\sqrt{n}/s_{\min}(X_n))^r\big]$, which is the form needed later when inverse-type quantities appear.
\begin{thm}\label{min_singular_value}
Let $X_n\in\mathbb{R}^{n\times p}$ be a random matrix with i.i.d.\ 
entries distributed as $\mathcal{N}(0,1)$. Fix a number $r\ge 0$, and 
assume that $n>p+r-1$. Then there exist (explicitly known) constants $c_1 (r),c_3 (r) >0$ and $c_2\in(0,1)$ such that
\[
\mathbb{E}\!\left[s_{\min}(X_n)^{-r}\right]
\le
\frac{c_1 (r)}{(n-p-1)^{r/2}}
+
\frac{c_3(r) \,c_2^{(n-p-r+1)/2}}{(n-p-1)^{(r-4)/2}} .
\]
In particular, if $p/n\to\gamma\in[0,1)$, then there exists a universal constant
$K_2(r)>0$ (depends only on $r$ and $\gamma$) such that
\[
\mathbb{E}\!\left[\left(\frac{\sqrt{n}}{s_{\min}(X_n)}\right)^r\right]\le K_2(r) .
\]
\end{thm}

\begin{proof}
    Let $W_n=X_n^\top X_n$. Then, by the relation between the smallest singular value of $X_n$ and the smallest eigenvalue of $W_n$, we have
\begin{align}\label{lemma min:1}
   \mathbb{E}\!\left[\left(\frac{1}{s_{\min}(X_n)}\right)^{r}\right] & =  \mathbb{E}\!\left[\left(\frac{1}{\lambda_{\min}(W_n)}\right)^{r/2}\right] \notag \\
   & = \frac{r}{2}\int_{0}^{\infty} t^{\frac{r}{2}-1} \mathbb{P}\left(\lambda_{max}(W_n^{-1})>t\right)\;dt.
\end{align}

Next, applying Lemma $5.2$ and Lemma $5.4$ of \cite{vershynin2011introductionnonasymptoticanalysisrandom}, we obtain
\begin{align}\label{lemma min:2}
    \mathbb{P}\left\{ \lambda_{\max}( W_n^{-1}) > t \right\} = \mathbb{P}\left\{ \underset{v: v \in \mathcal{S}^{p-1} }{\sup} ( v^\top W_n^{-1} v ) > t \right\} \leq 
    \mathbb{P}\left\{ \underset{u \in \mathcal{N}(\mathcal{S}^{p-1},\frac{1}{3}) }{\sup} ( u^\top W_n^{-1} u ) > t /3\right\} ,
\end{align}
where $\mathcal{N}(\mathcal{S}^{p-1},\frac{1}{3}) \leq 7^p$. For any $u$ such that $u \in \mathcal{S}^{p-1}$, $U=u^\top W_n^{-1} u \sim \text{Inverse-}\chi^2_{n - p +1}$ since $\|u\| = 1$. Therefore, using the union bound together with \eqref{lemma min:2}, we conclude that
\begin{align}\label{lemma min:4}
    \mathbb{P} \left( \lambda_{\max}( W_n^{-1}) > t \right)  \leq \mathcal{N}(\mathcal{S}^{p-1},\frac{1}{3})\,\mathbb{P} \left( U > t/3 \right)   \leq 7^p \mathbb{P} \left( U > t/3 \right).
\end{align}
\noindent Since $U$ has the same distribution as $U_1/2$, where $U_1 \sim \Ical \Gcal\{(n-p+1) / 2, 1 \}$, it follows that for $r \geq 2$,
\begin{align}\label{lemma min:5}
    \mathbb{P} \left( \lambda_{\max}( W_n^{-1}) > t \right) &\leq 
     7^p \mathbb{P} \left( U_1 > 2t/3 \right)\notag \\ 
     &\leq \frac{7^p}{\Gamma\{ (n-p+1) /2\} }\int_{2t/3}^{\infty} \frac{1}{x^{- \frac{n-p+1}{2} - 2}} e^{1/x} dx \nonumber\\
    &\stackrel{(a)}{\leq}  \frac{7^p}{\left( \frac{n-p-1} {2e}\right)^{\frac{n-p-1}{2}} \frac{\sqrt{\pi}}{3} } \int_{2t/3}^{\infty} \frac{1}{x^{- \frac{n-p+1}{2} - 2}} \frac{1}{x^2} dx \\
    &\leq \frac{3}{\sqrt{\pi}} \frac{7^p}{\left( \frac{(n-p-1)}{2e} \right)^{\frac{n-p-1}{2}}} \frac{3^{\frac{n-p-1}{2}}}{\left(2t \right)^{- \frac{n - p +1} {2}}} \int_{2t/3} ^{\infty} \frac{1}{x^2} dx \\
    &= \frac{3}{\sqrt{\pi}} \frac{7^p}{\left( \frac{(n-p-1)}{2e} \right)^{\frac{n-p-1}{2}}} 
    \left(\frac{3}{2t} \right)^{\frac{n-p+3}{2}} \\
    &= \frac{3}{\sqrt{\pi}} \frac{7^p}{\left( \frac{(n-p-1)}{2e} \right)^{\frac{n-p-1}{2}}} 
    \left(\frac{3}{2t} \right)^{\frac{n-p-r+1}{2}} \left(\frac{3}{2t} \right)^{\frac{r}{2}+1},  
\end{align}
where (a) follows from the lower bound for the Gamma function in \cite{Batir2017GammaBounds}, which is applicable since $n-p\ge 1$ as $r \geq 2$. Consequently, for a fixed constant $K$ such that $t \geq K / (n - p -1)$, we deduce from \eqref{lemma min:5} that
\begin{align}\label{lemma min:6}
    \mathbb{P} \left( \lambda_{\max}( W_n^{-1})> t \right) &\leq \frac{3}{\sqrt{\pi}} \frac{7^p}{\left( \frac{(n-p-1)}{2e} \right)^{\frac{n-p-1}{2}}} 
    \left(\frac{3 (n - p - 1)}{2K} \right)^{\frac{n-p-r+1}{2}} \left(\frac{3}{2t} \right)^{\frac{r}{2}+1} \nonumber\\
    &\leq \frac{3 (2e)^{\frac{r}{2}-1}} {\sqrt{\pi} (n - p - 1)^{\frac{r}{2}-1}}
     \left(  \frac{21 e} {K} \right)^{\frac{n-p-r+1}{2}} 
    \left(\frac{3}{2t} \right)^{\frac{r}{2}+1}.
\end{align}

\noindent
Combining \eqref{lemma min:6} with the integral representation \eqref{lemma min:1}, we obtain
\begin{align}\label{lemma min:7}
    \mathbb{E}\!\left[\left(\frac{1}{s_{\min}(X_n)}\right)^{r}\right]  &\leq 
    \frac{r}{2}\int_{0}^{\frac{K}{n-p-1}}  t^{\frac{r}{2}-1} \mathbb{P}\left(\lambda_{max}(W_n^{-1})>t\right)\;dt. + \nonumber\\
&\qquad \frac{3r (2e)^{\frac{r}{2}-1}} {2\sqrt{\pi} (n - p - 1)^{\frac{r}{2}-1}}
     \left(  \frac{21 e} {K} \right)^{\frac{n-p-r+1}{2}}   \int_{\frac{K}{n-p-1}}^\infty 
     t^{\frac{r}{2}-1} \left(\frac{3}{2t} \right)^{\frac{r}{2}+1} dt \nonumber\\
     &  \leq \frac{r}{2} \int_{0}^{\frac{K}{n-p-1}} t^{\frac{r}{2}-1} dt
     +
     \frac{27r (3e)^{\frac{r}{2}-1}} {8K\sqrt{\pi} (n - p - 1)^{\frac{r}{2}-1}}
     \left(  \frac{21 e} {K} \right)^{\frac{n-p-r+1}{2}}  (n - p -1 ) \nonumber\\
     &= \frac{ c_1 } {(n - p -1)^{\frac{r}{2}}} + \frac{c_2^{\frac{n-p-r+1}{2}}c_3 } {(n- p - 1)^{\frac{r}{2}-2}}, 
\end{align}
where $c_1 (r) =r K^{r/2}/2$, $c_2=(21e/K)$ and $c_3 (r)= \frac{27r (3e)^{\frac{r}{2}-1}} {8K\sqrt{\pi}}$. This establishes the first claim upon choosing $K>21e$.

For the second assertion, observe that \eqref{lemma min:7} implies
\begin{align}\label{lemma min:8}
   \mathbb{E}\!\left[\left(\frac{\sqrt{n}}{s_{\min}(X_n)}\right)^{r}\right]  \leq   \frac{ c_1 (r)} {(1 - \frac{p}{n} -\frac{1}{n})^{\frac{r}{2}}} + \frac{(n- p - 1)^{2}c_2^{\frac{n-p-r+1}{2}}c_3 (r) } {(1 - \frac{p}{n} -\frac{1}{n})^{\frac{r}{2}}}.
\end{align}
The desired uniform bound now follows for every $n$ since $p/n \to \gamma \in [0,1)$ as $n \to \infty$ and $c_2\in (0,1)$. This proves the result for $r\ge 2$; it then automatically follows for all
$0\le r\le 2$ by the standard monotonicity property of moments.
\end{proof}

\begin{remark}
Theorem \ref{min_singular_value} extends directly to the regime $p/n \to \gamma \in (1,\infty)$. 
The proof proceeds similarly as in part (i) of Lemma \ref{lem:double_descent} upon 
interchanging the roles of $n$ and $p$, and replacing $X_n$ by $X_n^\top$. Note that $X_n$ and $X_n^\top$ share the same singular values. Furthermore, invoking part (ii) of Lemma \ref{lem:double_descent}, 
one obtains that when $p/n \to 1$,
\[
\mathbb{E}\!\left[\left(\frac{\sqrt{n}}{s_{\min}(X_n)}\right)^{r}\right] \to \infty
\quad \text{for any } r > 0, \text{ as } n \to \infty.
\]
The proofs are omitted since they follow by arguments identical to those already presented.
\end{remark}

\noindent We next control the upper end of the singular-value spectrum. Lemma~\ref{max_singular_value} bounds positive moments of $s_{\max}(X_n)$ when $X_n$ has i.i.d.\ $\Ncal(0,1)$
entries. It is well known that $s_{\max}(X_n)$ typically grows on the order of
$\sqrt{n}+\sqrt{p}$ (\cite{vershynin2011introductionnonasymptoticanalysisrandom,bai_yin}). Lemma~\ref{max_singular_value} formalizes this in the context of moment convergence by showing
that, for any fixed $r>0$, the $r$-th moment of $s_{\max}(X_n)$ is controlled at
the corresponding order, namely $\EE[s_{\max}(X_n)^r]\lesssim (\sqrt{n}+\sqrt{p})^r$
(up to constants depending only on $r$). As a consequence, under the proportional
growth regime $p/n\to\gamma\in[0,\infty)$, the normalized quantity
$s_{\max}(X_n)/\sqrt{n}$ has uniformly bounded moments.

\begin{lem}\label{max_singular_value}
Let $X_n\in\mathbb{R}^{n\times p}$ be a random matrix with i.i.d.\ entries distributed as
$\mathcal{N}(0,1)$. Then, for any $r>0$, there exist (explicitly known) constants $\tilde{c_1} (r),\;\tilde{c_2} (r) >0$ such that
\[
\mathbb{E}\!\left[\left(s_{\max}(X_n)\right)^{r}\right]
\le
 \tilde{c_1}(r) (\sqrt{n}+\sqrt{p})^r + \tilde{c_2}(r).
\]
Moreover, if $p/n\to\gamma\in[0,\infty)$, then there exists a universal constant $K_3(r)>0$ (depends only on $r$ and $\gamma$)
such that
\[
 \mathbb{E}\!\left[\left(\frac{s_{\max}(X_n)}{\sqrt{n}}\right)^{r}\right] \le K_3(r) .
\]
\end{lem}

\begin{proof}
Since $s_{\max}(X_n)=\sqrt{\|X_n^\top X_n\|_2}$, we may write
\[
\mathbb{E}\!\left[\left(s_{\max}(X_n)\right)^{r}\right]=
\mathbb{E}\!\left[\|X_n^\top X_n\|_2^{\,\frac{r}{2}}\right]
=
\mathbb{E}\!\left[\|X_n\|_2^{\,r}\right]
= \frac{r}{2}\int_{0}^{\infty} t^{\frac{r}{2}-1}\, \mathbb{P}\!\left(\|X_n\|_2^{\,2} > t\right)\,dt.
\]
Fix any constant $C = \sqrt{2}$. Splitting the integral at $C^2(\sqrt{n}+\sqrt{p})^2$ yields
\begin{align} \label{lemma max:1}
\mathbb{E}\!\left[\left(s_{\max}(X_n)\right)^{r}\right]
&=\frac{r}{2}
\int_{0}^{C^2(\sqrt{n}+\sqrt{p})^2}
t^{\frac{r}{2}-1}\,\mathbb{P}\!\left(\|X_n\|_2^{\,2} > t\right)\,dt
+\frac{r}{2}
\int_{C^2(\sqrt{n}+\sqrt{p})^2}^{\infty}
t^{\frac{r}{2}-1}\mathbb{P}\!\left(\|X_n\|_2^{\,2} > t\right)\,dt \\
&\le
\frac{rC^r}{2}(\sqrt{n}+\sqrt{p})^r
+
\frac{r}{2}\int_{0}^{\infty}
t^{\frac{r}{2}-1}\,\mathbb{P}\!\left(\|X_n\|_2^{\,2} > t + C^2(\sqrt{n}+\sqrt{p})^2\right)\,dt \\
&\le
\frac{rC^r}{2}(\sqrt{n}+\sqrt{p})^r
+
\frac{r}{2}\int_{0}^{\infty} t^{\frac{r}{2}-1}
\mathbb{P}\!\left(
\|X_n\|_2^{\,2} >
\Bigg(\frac{C(\sqrt{n}+\sqrt{p})}{\sqrt{2}}+\sqrt{\frac{t}{2}}\Bigg)^2
\right)\,dt .
\end{align}
The last step follows from the inequality $(a+b)^2 \le 2(a^2+b^2)$ for $a,b>0$.

Now, by \cite[Corollary 5.35]{vershynin2011introductionnonasymptoticanalysisrandom}, there exists an appropriate choice of $C$ (independent of $n$) such that
\begin{align*}
\mathbb{P}\!\left(
\|X_n\|_2^{\,2} >
\Bigg(\frac{C(\sqrt{n}+\sqrt{p})}{\sqrt{2}}+\sqrt{\frac{t}{2}}\Bigg)^2
\right)
&=
\mathbb{P}\!\left(
\|X_n\|_2 >
\frac{C}{\sqrt{2}}\Big(\sqrt{n}+\sqrt{p}\Big) + \sqrt{\frac{t}{2}}
\right) \\
&\le
2\exp\!\left(-c^2 t\right),
\end{align*}
where $c=0.25$. Substituting this bound into \eqref{lemma max:1} gives
\begin{align*}
\mathbb{E}\!\left[\left(s_{\max}(X_n)\right)^{r}\right]
&\le
\frac{rC^r}{2}(\sqrt{n}+\sqrt{p})^r
+
r \int_{0}^{\infty} t^{\frac{r}{2}-1} \exp\!\left(-c^2 t\right)\,dt \\
&=
\frac{rC^r}{2}(\sqrt{n}+\sqrt{p})^r + r \Gamma\!\left(\frac{r}{2}\right)c^{-r}\\
&=  \tilde{c_1}(r) (\sqrt{n}+\sqrt{p})^r +  \tilde{c_2}(r),
\end{align*}
where $ \tilde{c_1}(r)=\frac{rC^r}{2}$ and $ \tilde{c_2}(r)= r \Gamma\!\left(\frac{r}{2}\right)c^{-r}$, which proves the first claim.

For the second claim, dividing both sides by $n^{r/2}$ yields
\begin{align*}
   \mathbb{E}\!\left[\left(\frac{s_{\max}(X_n)}{\sqrt{n}}\right)^{r}\right] \leq \tilde{c_1}(r) \left(1+\sqrt{\frac{p}{n}}\right)^r + \frac{\tilde{c_2}(r)}{n^{\frac{r}{2}}}.
\end{align*}
The desired uniform bound follows for every $n$ since $p/n \to \gamma \in [0,\infty)$ as $n \to \infty$.
\end{proof}

\begin{remark} \label{remark:subg}
\noindent
Lemma~\ref{max_singular_value} extends beyond the Gaussian setting. In particular, an analogous bound remains valid when the entries of $X_n$ are sub-Gaussian. This follows from standard nonasymptotic results for sub-Gaussian random matrices; see, for example, \cite[Corollary 5.35]{vershynin2011introductionnonasymptoticanalysisrandom} and its general formulation in \cite[Theorem 5.39]{vershynin2011introductionnonasymptoticanalysisrandom}. The primary modification required in the proof concerns the constants ($c$ and $C$), which now depend on the sub-Gaussian norm of the underlying distribution (see Section~\ref{sec_subg}). In contrast, Theorem \ref{theorem_condition_number} and Theorem~\ref{min_singular_value} are essentially Gaussian-specific and does not, in general, extend to arbitrary sub-Gaussian designs. We refer the reader to Section~\ref{sec_subg} for further discussion and an explicit counterexample.
\end{remark}

\begin{remark}
\noindent
Theorem~\ref{min_singular_value} and Lemma~\ref{max_singular_value} can be combined to recover results of the type stated in Theorem~\ref{theorem_condition_number} under an analogous asymptotic regime. Indeed, for any $r > 0$,
\[
\{\kappa(Z_n)\}^r
=
\left(\frac{s_{\max}(Z_n)}{\sqrt{n}}\right)^r
\left(\frac{\sqrt{n}}{s_{\min}(Z_n)}\right)^r.
\]
Consequently, moment bounds for the condition number follow directly from the corresponding bounds on the extreme singular values via an application of the Cauchy--Schwarz inequality. We omit the explicit derivation, as the argument is immediate from this decomposition.
\end{remark}

\section{Applications to High-Dimensional Statistics: Spectral Properties and Moments of Gram Matrices}\label{sec:gram_app}
\noindent In this section, we show how the moment bounds for various spectral quantities established in Section~\ref{sec:main_results} translate into concrete implications for the spectrum of Gram matrices arising from random designs across different statistical settings. In particular, if $X_n$ is an $n \times p$ random \emph{design matrix} whose rows are mutually independent Gaussian vectors, then the eigenvalues and condition number of the associated \emph{Gram matrix} $S_n := \frac{1}{n} X_n^\top X_n$ are closely tied to the singular values and condition number of $X_n$ itself. As we demonstrate below, these spectral quantities play a central role in characterizing the statistical stability of several estimators (Theorems~\ref{thm:ridge_pred_bound_hd_multi} and \ref{thm:inv_cov_error_rate}). They also govern the algorithmic behavior of iterative optimization methods---for example, the convergence rates of first-order procedures and Krylov-subspace algorithms when applied to the corresponding normal equations (Theorem~\ref{thm:gd_gram_condition_number}). The following subsections formalize these consequences. 

\subsection{Prediction Risk for ridge regression with random predictors} \label{sec:ridge}

\noindent
Consider the multi-response linear model with $q$ responses, $p$ predictors, and $n$ samples, given by 
\begin{equation}\label{model}
 Y_n = X_n B + E_n,
\end{equation}
where $Y_n\in\mathbb{R}^{n\times q}$ is the response matrix, $X_n\in \mathbb{R}^{n\times p}$ is the (random) design matrix, $E_n \in \mathbb{R}^{n\times q}$ is the error matrix, and $B\in\mathbb{R}^{p\times q}$ is the coefficient matrix. The matrices $X_n$ and $E_n$ are assumed to be independent, and the rows of $E_n$ are assumed to be independent and identically distributed as $N(0,\Sigma_\varepsilon)$. 

We work in a high-dimensional regime in which the number of predictors $p$ and the number of responses $q$ may both depend on $n$ and increase with $n$. Technically, the parameters $B$ and $\Sigma_\varepsilon$ also depend on $n$ in this regime, but we suppress this dependence for clarity and to distinguish these parameters from data-based quantities. Let
\[
S_n = \frac{1}{n}X_n^\top X_n
\]
denote the sample Gram matrix. For $\lambda>0$, the ridge estimator for $B$ is given by
\[
\hat B_\lambda
= (X_n^\top X_n + \lambda I_p)^{-1} X_n^\top Y_n =  
\left( S_n+ \tilde{\lambda}_n I_p \right)^{-1}\,\frac{1}{n}X_n^\top Y_n,
\]
where $\tilde{\lambda}_n=\frac{\lambda}{n}$. Conditionally on $X_n$, the prediction risk of ridge regression is defined as
\begin{align}\label{pred_risk}
R_{\mathrm{pred}}(\tilde{\lambda}_n \mid X_n)
\;:=\;
\frac{1}{n}\,
\mathbb{E}\big[\|X_n(\hat B_\lambda - B)\|_F^2 \,\big|\, X_n\big]
\;=\;
\mathbb{E}\big[\mathrm{tr}\{(\hat B_\lambda - B)^\top S_n (\hat B_\lambda - B)\}
\,\big|\, X_n\big].    
\end{align}
For the case $q = 1$, corresponding to simple linear regression, the high-dimensional asymptotic behavior of $R_{\mathrm{pred}}(\tilde{\lambda}_n \mid X_n)$ has been extensively studied in the literature; see, for example, \cite{ridge1,ridge2,Hastie}. Under the random design framework, \cite{Hastie} derived the almost sure limit of the quantity $R_{\mathrm{pred}}(\tilde{\lambda}_n \mid X_n)$  for the case $q = 1$ (see \cite[Proposition 2]{Hastie}). However, this almost sure convergence does not provide information about the expected prediction risk, where the expectation is taken over the randomness of $X_n$, which is often of interest in the literature (\cite{ridge1}). See Remark \ref{remark:ridgeless} for details.

That said, we derive non-asymptotic bounds for the expected prediction risk \emph{for general $q$} in a high-dimensional regime. We begin with an intermediate lemma that serves as a stepping stone toward this result. This lemma shows that the prediction risk can be decomposed into a ``bias'' term and a ``variance'' term, and provides upper bounds for both quantities in terms of the largest and smallest eigenvalues of the Gram matrix $S_n$. The proof of this Lemma is deferred to Appendix \ref{apnd_A}.

\begin{lem}[Bias--variance decomposition and bounds for multi-response ridge prediction risk]\label{lem:multi_variance_term}
Consider the multi-response linear model defined in (\ref{model}). Then, for any
fixed design $X_n$ and any $\lambda>0$,
\begin{eqnarray*}
& & R_{\mathrm{pred}}(\tilde{\lambda}_n\mid X_n)\\
&=& \tilde{\lambda}_n^2\,\operatorname{tr}\!\Big(B^\top (S_n+\tilde{\lambda}_n I_p)^{-1}S_n(S_n+\tilde{\lambda}_n I_p)^{-1}B\Big)
+
\frac{\mathrm{tr}(\Sigma_\varepsilon)}{n}\,
\operatorname{tr}\!\Big(S_n(S_n+\tilde{\lambda}_n I_p)^{-1}S_n(S_n+\tilde{\lambda}_n I_p)^{-1}\Big).
\end{eqnarray*}
Moreover,
\[
\tilde{\lambda}_n^2\,\operatorname{tr}\!\Big(B^\top (S_n+\tilde{\lambda}_n I_p)^{-1}S_n(S_n+\tilde{\lambda}_n I_p)^{-1}B\Big)
\le
\frac{\tilde{\lambda}_n^2\,\lambda_{\max}(S_n)}{(\lambda_{\min}(S_n)+\tilde{\lambda}_n)^2}\,\|B\|_F^2,
\]
and
\[
\frac{\mathrm{tr}(\Sigma_\varepsilon)}{n}\,
\operatorname{tr}\!\Big(S_n(S_n+\tilde{\lambda}_n I_p)^{-1}S_n(S_n+\tilde{\lambda}_n I_p)^{-1}\Big)
\le
\frac{p\,\mathrm{tr}(\Sigma_\varepsilon)}{n}\,
\frac{\lambda_{\max}(S_n)^2}{(\lambda_{\min}(S_n)+\tilde{\lambda}_n)^2}.
\]
\end{lem}

\noindent
We now show how the results from Section \ref{sec:main_results} can be applied to bound the moments of $\lambda_{\max}(S_n)$ and $\lambda_{\min}(S_n)$, and thereby obtain non-asymptotic upper bounds for the expected prediction risk. Moreover, these bounds imply that, under mild regularity conditions, the expected prediction risk converges to zero as $n \to \infty$. 
\begin{thm}[Mean prediction risk in high dimensions]\label{thm:ridge_pred_bound_hd_multi}
Consider the random-design multi-response linear model in \eqref{model} and
define $R_{\mathrm{pred}}(\tilde{\lambda}_n\mid X_n)$ as in \eqref{pred_risk}. Assume that
the rows of $X_n\in\mathbb{R}^{n\times p}$ are independent and identically
distributed as $N(0,\tilde\Sigma_n)$, where $\tilde\Sigma_n$ satisfies
\eqref{cond-const}. Assume further that the rows of $E_n$ are independent and
identically distributed as $N(0,\Sigma_\varepsilon)$, where $\Sigma_\varepsilon$
satisfies \eqref{cond-const} with corresponding lower bound $c_{\varepsilon,m}$ and upper bound $c_{\varepsilon,M}$.

Then for a fixed $\lambda>0$ there exists a constant $K>0$ (depending only on $\lambda,c_m,c_M,
c_{\varepsilon,m},c_{\varepsilon,M}$) such that, for all
sufficiently large $n$,
\[
\mathbb{E}\big[ R_{\mathrm{pred}}(\lambda\mid X_n)\big]
\le
K\Big(\tilde{\lambda}_n^2 \|B\|_F^2 + \frac{pq}{n}\Big).
\]
In particular, if $\|B\|_F/n \to 0$ and $pq/n\to0$ as $n\to\infty$, then
$\mathbb{E}\big[ R_{\mathrm{pred}}(\tilde{\lambda}_n\mid X_n)\big]\to 0$ as $n \to \infty$.
\end{thm}

\begin{proof}
By Lemma \ref{lem:multi_variance_term}, for any fixed design $X_n$ and any
$\lambda>0$,
\begin{align}\label{eq:ridge_upper_from_lemma_multi}
R_{\mathrm{pred}}(\tilde{\lambda}_n\mid X_n)
\le
\frac{1}{(\lambda_{\min}(S_n)+\tilde{\lambda}_n)^2}
\Big\{
\tilde{\lambda}_n^2 \lambda_{\max}(S_n)\|B\|_F^2
+
\frac{p\,\mathrm{tr}(\Sigma_\varepsilon)}{n}\, \,\lambda_{\max}(S_n)^2
\Big\}.
\end{align}

We next control $(\lambda_{\min}(S_n)+\tilde{\lambda}_n)^{-2}$ and $\lambda_{\max}(S_n)$ in
terms of the singular values of a standard Gaussian matrix. Write
\[
X_n = Z_n \tilde\Sigma_n^{1/2},
\]
where $Z_n$ has i.i.d.\ $\mathcal{N}(0,1)$ entries. Then
\[
S_n=\frac{1}{n}\tilde\Sigma_n^{1/2}Z_n^\top Z_n\tilde\Sigma_n^{1/2}.
\]
Using $\lambda_{\max}(ABA)\le \|A\|_2^2\lambda_{\max}(B)$ and
$\lambda_{\min}(ABA)\ge s_{\min}(A)^2\lambda_{\min}(B)$ for $A\succeq 0$,
together with \eqref{cond-const}, we obtain
\begin{align}\label{eq:eig_sandwich_multi}
\lambda_{\max}(S_n)
\le
\lambda_{\max}(\tilde\Sigma_n)\,\lambda_{\max}\!\left(\frac{1}{n}Z_n^\top Z_n\right)
\le
c_M\left(\frac{s_{\max}(Z_n)}{\sqrt{n}}\right)^2,
\end{align}
and, for $n>p$ (so that $Z_n$ has full column rank almost surely),
\begin{align}\label{eq:eig_sandwich_min_multi}
\lambda_{\min}(S_n)
\ge
\lambda_{\min}(\tilde\Sigma_n)\,\lambda_{\min}\!\left(\frac{1}{n}Z_n^\top Z_n\right)
=
\lambda_{\min}(\tilde\Sigma_n)\left(\frac{s_{\min}(Z_n)}{\sqrt{n}}\right)^2
\ge
c_m\left(\frac{s_{\min}(Z_n)}{\sqrt{n}}\right)^2 .
\end{align}

Hence, using the inequality $(a+\lambda)^{-2}\le a^{-2}$ for $a>0$, we have
\[
\frac{1}{(\lambda_{\min}(S_n)+\tilde{\lambda}_n)^2}\le \frac{1}{(\lambda_{\min}(S_n))^2}.
\]
Using \eqref{eq:ridge_upper_from_lemma_multi} yields
\begin{align}\label{eq:ridge_expect_split_multi}
\mathbb{E}\big[R_{\mathrm{pred}}(\tilde{\lambda}_n\mid X_n)\big]
&\le
\tilde{\lambda}_n^2 \|B\|_F^2 \,\mathbb{E}\!\left[\frac{\lambda_{\max}(S_n)}{(\lambda_{\min}(S_n))^2}\right]
+
\frac{p}{n}\,\mathbb{E}\!\left[\mathrm{tr}(\Sigma_\varepsilon)\,
\frac{(\lambda_{\max}(S_n))^2}{(\lambda_{\min}(S_n))^2}\right].
\end{align}
Since $\Sigma_\varepsilon$ is deterministic and $\mathrm{tr}(\Sigma_\varepsilon)\le
q\,\lambda_{\max}(\Sigma_\varepsilon)$, it follows that
\[
\mathrm{tr}(\Sigma_\varepsilon)\le q\,c_{\varepsilon,M}.
\]
Moreover,
\[
\left(\frac{\lambda_{\max}(S_n)}{\lambda_{\min}(S_n)}\right)^2
=
\left(\frac{s_{\max}(X_n)}{s_{\min}(X_n)}\right)^4
=
(\kappa(X_n))^4.
\]
Therefore,
\begin{align}\label{eq:ridge_expect_split_multi2}
\mathbb{E}\big[R_{\mathrm{pred}}(\tilde{\lambda}_n\mid X_n)\big]
&\le
\tilde{\lambda}_n^2 \|B\|_F^2\,\mathbb{E}\!\left[\frac{\lambda_{\max}(S_n)}{(\lambda_{\min}(S_n))^2}\right]
+
c_{\varepsilon,M}\frac{pq}{n}\,\mathbb{E}\!\left[(\kappa(X_n))^4\right].
\end{align}

Using \eqref{eq:eig_sandwich_multi}--\eqref{eq:eig_sandwich_min_multi}, we obtain
\[
\frac{\lambda_{\max}(S_n)}{(\lambda_{\min}(S_n))^2}
\le
\frac{c_M}{c_m^2}\,
\left(\frac{s_{\max}(Z_n)}{\sqrt{n}}\right)^2
\left(\frac{\sqrt{n}}{s_{\min}(Z_n)}\right)^4.
\]
Applying the Cauchy--Schwarz inequality yields
\begin{align*}
\mathbb{E}\!\left[\frac{\lambda_{\max}(S_n)}{(\lambda_{\min}(S_n))^2}\right]
&\le
\frac{c_M}{c_m^2}\,
\Bigg\{\mathbb{E}\!\left[\left(\frac{s_{\max}(Z_n)}{\sqrt{n}}\right)^4\right]\Bigg\}^{1/2}
\Bigg\{\mathbb{E}\!\left[\left(\frac{\sqrt{n}}{s_{\min}(Z_n)}\right)^8\right]\Bigg\}^{1/2}.
\end{align*}

Since $pq/n\to0$, we have $p/n\to 0$
in particular, and hence $n>p+8-1$ for sufficiently large $n$. Therefore, by Theorem \ref{min_singular_value}
and Lemma \ref{max_singular_value} (with $r=8$ and $r=4$, respectively), there
exist positive constants $K_{\max}$ and $K_{\min}$ such that
\[
\sup_{n}\mathbb{E}\!\left[\left(\frac{s_{\max}(Z_n)}{\sqrt{n}}\right)^4\right]\le K_{\max},
\qquad
\sup_{n}\mathbb{E}\!\left[\left(\frac{\sqrt{n}}{s_{\min}(Z_n)}\right)^8\right]\le K_{\min}.
\]
Hence, using the above bounds together with Theorem \ref{theorem_condition_number},
it follows that there exists a constant $C>0$ (depending only on $c_m,c_M$ and
$\gamma$) such that, for all sufficiently large $n$,
\[
\mathbb{E}\!\left[\frac{\lambda_{\max}(S_n)}{(\lambda_{\min}(S_n))^2}\right]\le C,
\qquad
\mathbb{E}\!\left[(\kappa(X_n))^4\right]\le C.
\]
Substituting these bounds into \eqref{eq:ridge_expect_split_multi2} gives
\[
\mathbb{E}\big[R_{\mathrm{pred}}(\tilde{\lambda}_n\mid X_n)\big]
\le
C\tilde{\lambda}_n^2 \|B\|_F^2 + C c_{\varepsilon,M}\frac{pq}{n}
\]
Absorbing $\lambda^2$ and $c_{\varepsilon,M}$ into the generic constant yields the desired bound, thereby completing the proof of the first claim. The second claim follows upon setting $\tilde{\lambda}_n = \lambda / n$ in the preceding expression and using the assumption that $\|B\|_F/ n \to 0$ as $n \to \infty$.
\end{proof}

\begin{remark}
The condition $pq = o(n)$ (or equivalently $p = o(n)$ when $q$ is fixed) is standard in establishing asymptotic results for least squares type estimators; see, for example, \cite{armagan,bai_ghosh,sarkar}. This requirement is also natural from a dimensional scaling perspective. In the absence of any structural assumptions, such as low-rankness or sparsity, the squared Frobenius norm of the coefficient matrix scales proportionally with $pq$, reflecting the total number of free parameters. Consequently, the condition $pq = o(n)$ is both intuitively justified and consistent with prevailing assumptions in the existing literature.
\end{remark}

\begin{remark}
Following \cite{ridge2}, one may alternatively model the regression coefficient matrix $B$ as random. Suppose that $\mathbb{E}[\|B\|_F^2] = {B^*_n}^2$ and that $B^*_n / n \to 0$ as $n \to \infty$. Then the preceding argument carries over unchanged by conditioning on both $B$ and $X_n$, and consequently the result continues to hold when the coefficient matrix is random.

In general, the condition $\|B\|_F / n \to 0$ is mild and is satisfied under a broad class of scaling regimes; in particular, it holds whenever $\|B\|_F = O(pq)$. For instance, if the entries of $B$ are uniformly bounded (e.g. \cite{bai_ghosh}), then $\|B\|_F^2 \le C pq$ for some constant $C > 0$. Alternatively, consider a random coefficient model with $\mathbb{E}[\mathrm{vec}(B)] = 0$ and $\mathrm{Var}(\mathrm{vec}(B)) = \alpha^2 (pq)^{-1} I_{pq}$, where $\alpha > 0$ is a constant. In this case, $\mathbb{E}[\|B\|_F^2] = O(1)$. A similar assumption appears in \cite{ridge2} for the special case $q = 1$.
\end{remark}

\begin{remark}\label{remark:ridgeless}
In \cite[Proposition 2]{Hastie}, the authors establish an almost sure limit for $R_{\mathrm{pred}}(\tilde{\lambda}_n \mid X_n)$ in the case $q = 1$ under the regime $p/n \to \gamma \in [0,1)$, while imposing substantially weaker assumptions than Gaussianity. Specifically, in our notation they consider the design
$
X_n = Z_n \tilde{\Sigma}_n^{1/2},
$
where $Z_n$ is an $n \times p$ random matrix with i.i.d.\ entries having zero mean, unit variance, and a finite fourth moment, and where $\tilde{\Sigma}_n$ satisfies \eqref{cond-const}. Although this assumption is weaker than the Gaussian design we adopt to prove Theorem~\ref{thm:ridge_pred_bound_hd_multi}, it is important to emphasize that our objective is different. Rather than almost sure convergence, we require convergence of the \emph{mean} prediction risk, i.e., after taking expectation over the randomness of $X_n$. This constitutes a stronger requirement than the almost sure statement in \cite[Proposition 2]{Hastie}.

Indeed, removing the Gaussianity assumption from Theorem~\ref{thm:ridge_pred_bound_hd_multi} is not always possible, as we demonstrate in Section~\ref{sec_subg}. In particular, we construct a sub-Gaussian counterexample (see Proposition \ref{prop:subg_counterexample_r1}) satisfying all the assumptions of \cite[Proposition 2]{Hastie}, yet for which the expectation of the inverse of the minimum singular value diverges to infinity, a quantity that plays a key role in bounding the expected prediction risk. At the same time, the corresponding almost sure convergence may still hold by the Bai--Yin law \cite{bai_yin}, yielding a result consistent with \cite[Proposition 2]{Hastie}.
\end{remark}

\subsection{Expected estimation risk for inverse covariance matrices} \label{sec:cov}

\noindent
Another widely used Gram matrix in statistical analysis is the sample variance--covariance matrix. Let
$Y_1,\ldots,Y_n\in\mathbb{R}^p$ be i.i.d.\ centered Gaussian random vectors with
covariance matrix $\Sigma$, and define the sample covariance matrix
\[
S_n=\frac{1}{n}\sum_{i=1}^n Y_iY_i^\top .
\]
It is well known that $S_n$ is an unbiased estimator of $\Sigma$. In the
high-dimensional regime where $p$ grows with $n$ (typically with
$p/n\to\gamma\in[0,1)$), a large literature has developed sharp concentration
inequalities for $\|S_n-\Sigma\|_2$ under sub-Gaussian assumptions; see, for
example, \cite{vershynin2018high,wainwright2019high}. These results yield
high-probability rates of convergence and form the basis of many modern
nonasymptotic analyses. Under additional regularity conditions (for instance,
when the eigenvalues of $\Sigma$ are uniformly bounded away from $0$ and
infinity), the same concentration machinery can also be used to obtain
high-probability bounds for the inverse error $\|S_n^{-1}-\Sigma^{-1}\|_2$.

While such tail bounds are highly informative, they do not directly quantify
the \emph{average} performance of covariance estimation. In many applications,
one is instead interested in expectation- or moment-type quantities, such as
$\mathbb{E}[\|S_n-\Sigma\|_2^2]$, which provide a stronger notion of stability by
capturing both typical behavior and the contribution of rare extreme events.
In this direction, \cite{gram} established nonasymptotic bounds for moments of
$\|S_n-\Sigma\|_2$ for general \emph{sub-Gaussian} random vectors (with the
Gaussian case as a special instance), implying that for any fixed $r\ge 1$,

\begin{equation}\label{moment_bound_S}
 \big(\mathbb{E}\|S_n-\Sigma\|_2^r\big)^{1/r}
=
O\Big(\sqrt{\frac{p}{n}}\vee\frac{p}{n}\Big),   
\end{equation}
under uniformly bounded eigenvalues of $\Sigma$.

A related and practically important problem is estimation of the precision
matrix $\Sigma^{-1}$ by the empirical inverse $S_n^{-1}$ (\cite{Muirhead1982Aspects,Anderson2003Multivariate}), which appears frequently in multivariate analysis and high-dimensional inference. For
instance, in sparse precision matrix estimation,
\cite{Peng2009PartialCorrJointSparse} considered the regime $p<n$ and proposed
estimating diagonal elements of $\Sigma^{-1}$ using the corresponding diagonal
entries of $S_n^{-1}$. However, inverse covariance estimation is substantially
more delicate: it is inherently unstable (\cite{LedoitWolf2004WellConditioned}) whenever $\lambda_{\min}(S)$ is small,
since this may lead to large fluctuations in $S_n^{-1}$. This motivates the study
of moment bounds for inverse estimation errors, such as
$\mathbb{E}[\|S_n^{-1}-\Sigma^{-1}\|_2^2]$. Since $p$ is allowed to increase with $n$, the dimension of the parameter $\Sigma$ also changes with the sample size. As noted earlier in Section~\ref{sec:ridge}, we suppress this dependence on $n$ in our notation for clarity and simplicity of exposition. 

In the next theorem, we show that under Gaussianity and the same bounded
eigenvalue assumptions as in \eqref{cond-const} on $\Sigma$, for any fixed $r\ge 1$ the inverse covariance estimation error i.e. $\mathbb{E}[\|S_n^{-1}-\Sigma^{-1}\|_2^r]$ achieves the same rate as the corresponding moment bound for $\mathbb{E}[\|S_n-\Sigma\|_2^r]$. This yields a strong stability for inverse covariance estimation in the Gaussian case. In the subsequent section (see Proposition~\ref{prop:inv_cov_infinite_mean}), we show that this behavior is specific to Gaussian designs(\cite{gram}) and, unlike the moment bounds for
$\|S_n-\Sigma\|_2$, need not hold for general sub-Gaussian distributions.

\begin{thm}[Moment bound for inverse sample covariance error]\label{thm:inv_cov_error_rate}
Let $Y_1,\ldots,Y_n\in\mathbb{R}^p$ be i.i.d.\ centered Gaussian random vectors
with covariance matrix $\Sigma$, and define
\[
S_n:=\frac{1}{n}\sum_{i=1}^n Y_iY_i^\top .
\]
Assume that $\Sigma$ satisfies \eqref{cond-const}.
Fix $r\ge 0$ and assume $n>p+2r-1$. Then there exists a constant $C_r>0$
(depending only on $r,c_m,c_M$) such that 
\[
\Big(\mathbb{E}\big[\|S_n^{-1}-\Sigma^{-1}\|_2^r\big]\Big)^{1/r}
\le
C_r\left(\sqrt{\frac{p}{n}}\vee\frac{p}{n}\right).  
\]
\end{thm}

\begin{proof}
Using the identity $A^{-1}-B^{-1}=A^{-1}(B-A)B^{-1}$ with $A=S_n$ and $B=\Sigma$,
\[
S_n^{-1}-\Sigma^{-1}=S_n^{-1}(\Sigma-S_n)\Sigma^{-1}.
\]
Therefore,
\[
\|S_n^{-1}-\Sigma^{-1}\|_2
\le
\|S_n^{-1}\|_2\,\|\Sigma-S_n\|_2\,\|\Sigma^{-1}\|_2
\le
c_m^{-1}\,\|S_n^{-1}\|_2\,\|\Sigma-S_n\|_2.
\]
Raising both sides to the power $r$ and applying Cauchy--Schwarz gives
\begin{align}\label{eq:inv_cov_holder}
\mathbb{E}\big[\|S_n^{-1}-\Sigma^{-1}\|_2^r\big]
&\le
c_m^{-r}
\Big\{\mathbb{E}\|S_n^{-1}\|_2^{2r}\Big\}^{1/2}
\Big\{\mathbb{E}\|\Sigma-S_n\|_2^{2r}\Big\}^{1/2}.
\end{align}

Write $Y_i=\Sigma^{1/2}Z_i$ with $Z_i\sim N(0,I_p)$ i.i.d., and let
$Z\in\mathbb{R}^{n\times p}$ be the matrix with rows $Z_i^\top$. Then
\[
S_n=\Sigma^{1/2}\Big(\frac{1}{n}Z^\top Z\Big)\Sigma^{1/2}.
\]
Since $\lambda_{\min}(S_n)\ge \lambda_{\min}(\Sigma)\lambda_{\min}((1/n)Z^\top Z)$,
it follows that
\[
\|S_n^{-1}\|_2=\frac{1}{\lambda_{\min}(S_n)}
\le
c_m^{-1}\left(\frac{\sqrt{n}}{s_{\min}(Z)}\right)^2,
\]
and hence
\[
\mathbb{E}\|S_n^{-1}\|_2^{2r}
\le
c_m^{-2r}\,
\mathbb{E}\!\left[\left(\frac{\sqrt{n}}{s_{\min}(Z)}\right)^{4r}\right].
\]
By Theorem \ref{min_singular_value} (applied with exponent $4r$), the right-hand
side is finite under $n>p+4r-1$, and therefore
\[
\Big\{\mathbb{E}\|S_n^{-1}\|_2^{2r}\Big\}^{1/2}\le C_{r,1}
\]
for a constant $C_{r,1}$ depending only on $r$ and $c_m$.

To bound $\mathbb{E}\|\Sigma-S_n\|_2^{2r}$, we use the nonasymptotic operator norm
bound for the sample covariance matrix from \cite{gram}. In particular, it
follows from \cite[Corollary 2]{gram} that there exists an
absolute constant $C_{r,2}>0$ depending only on $r$ and $c_M$  such that
\[
\Big(\mathbb{E}\|\Sigma-S_n\|_2^{2r}\Big)^{1/(2r)}
\le
C_{r,2}\|\Sigma\|_2\left(\sqrt{\frac{p}{n}}\vee\frac{p}{n}\right).
\]
Substituting these bounds into \eqref{eq:inv_cov_holder} gives
\[
\Big(\mathbb{E}\big[\|S_n^{-1}-\Sigma^{-1}\|_2^r\big]\Big)^{1/r}
\le
C_r\left(\sqrt{\frac{p}{n}}\vee\frac{p}{n}\right),
\]
where $C_r$ depends only on $r,c_m$ and $c_M$. This completes the proof.

\end{proof}

\subsection{Solving Gram linear systems: Expected iteration complexity and the role of condition number moments}
\label{sec:optim}

\noindent Consider the linear system in the unknown parameter $\theta\in\mathbb{R}^p$,
\begin{equation}\label{eq:gram_system}
S_n\theta=b,
\end{equation}
where $b\in\mathbb{R}^p$ is a known vector. Here $S_n$ denotes the sample Gram
matrix associated with a (random) design matrix $X_n\in\mathbb{R}^{n\times p}$,
defined by
\[
S_n=\frac{1}{n}X_n^\top X_n.
\]
Systems of the form \eqref{eq:gram_system} arise frequently in statistics and
machine learning, for example through normal equations, iterative least-squares
subproblems, and kernel methods. In such settings, the condition number plays a
fundamental role because the righthand side vector $b$ typically depends on the
observed data. Consequently, it is important to understand how perturbations or
noise in the data, which directly induce perturbations in $b$, affect the
optimal solution. In particular, as discussed by
\cite{poggio2020doubledescentconditionnumber}, the condition number
$\kappa(X_n)$ governs the sensitivity of the solution to changes in $b$, and
hence directly influences stability and generalization. A representative example
comes from supervised learning, where the goal is to estimate an unknown
function from measurements at random points in a high-dimensional space, and
where low sensitivity to errors in the data is essential for good predictive
performance. In this section, we show that studying condition numbers is
therefore unavoidable for understanding the computational behavior and
convergence properties of algorithms used to solve such Gram-type linear
systems.
In the growing-dimensional regime where $p/n \to \gamma \in [0,1]$ and $p \le n$, the Gram matrix satisfies $S_n \succ 0$. Consequently, the linear system in \eqref{eq:gram_system} admits a unique solution given by $\theta_n^\star = S_n^{-1}b.$

In contrast, when $p/n \to \gamma \in [1,\infty)$ and $p>n$, the system may fail to admit a solution unless
$b \in \operatorname{range}(X_n^\top)$ (or equivalently $\operatorname{range}(S_n)$). This condition holds in many statistical settings---for instance, in least-squares normal equations arising from linear regression. Even when a solution exists, it is generally not unique. Among all solutions, to obtain a unique minimum-$\ell_2$-norm solution, we define $\theta_n^\star = S_n^{+}b$, where $S_n^{+}$ denotes the (unique) Moore--Penrose pseudo-inverse of $S_n$ \cite{Penrose_1955}. To summarize, the definition of $\theta_n^\star$ across dimensional regimes can be written as
\[
\theta_n^\star =
\begin{cases}
S_n^{-1} b, 
& \text{if } p \le n, \\[8pt]
S_n^{+} b,
& \text{if } p > n.
\end{cases}
\]

A convenient construction of $S_n^{+}$ follows from the singular value decomposition. Suppose $\operatorname{rank}(S_n)=r$, and let
$S_n = UDV^\top$ denote the singular value decomposition (SVD) of $S_n$, where $U$ and $V$ are semi-orthogonal matrices of appropriate dimensions and
$D\in\mathbb{R}^{r\times r}$ is diagonal containing the nonzero singular values of $S_n$. Then the Moore--Penrose pseudo-inverse is uniquely given by $S_n^{+} = UD^{-1}V^\top$.

\medskip

\noindent
{\bf An iterative gradient descent approach} Although $\theta_n^\star$ admits a closed-form representation, computing it directly becomes impractical when $p$ is large, since it requires forming and factorizing the $p\times p$ matrix $S_n$. This motivates the use of iterative methods--under the assumption $b\in \operatorname{range}(X_n^\top)$--which avoid explicit inversion and instead exploit matrix-vector products; see, for example, \cite{saad2003iterative,trefethen1997numerical}.

A standard viewpoint is to interpret \eqref{eq:gram_system} as the first-order
optimality condition of a convex quadratic optimization problem. Define
\begin{equation}\label{eq:quad_obj}
g_n(\theta):=\frac{1}{2}\theta^\top S_n\theta-b^\top\theta,
\qquad \theta\in\mathbb{R}^p,
\end{equation}
so that
\[
\nabla g_n(\theta)=S_n\theta-b,
\qquad
\nabla^2 g_n(\theta)=S_n.
\]
Then $\theta_n^\star$ is equivalently the unique minimizer of $g_n$. In
particular, we consider gradient descent iterates $\{\theta^{(t)}\}_{t\ge 0}$
initialized at $\theta^{(0)}$ and updated according to
\[
\theta^{(t+1)}=\theta^{(t)}-\eta\,\nabla g_n(\theta^{(t)}),
\]
for a step size $\eta>0$. The following lemma establishes linear convergence of gradient descent, with a rate depending only on the condition number of $X_n$, provided that $\theta^{(0)} \in \operatorname{range}(S_n)$ and the step size $\eta$ is chosen appropriately. While related results appear in \cite{saad2003iterative,trefethen1997numerical,Hastie}, the corresponding bounds are not directly comparable to ours, and a nontrivial modification would be required to translate their arguments to the present setting. To avoid these additional complications, we provide an independent proof tailored to our framework, which will be instrumental for the subsequent analysis. The proof is deferred to Appendix \ref{apnd_A}.

\begin{lem}[Linear convergence of gradient descent on $\operatorname{range}(S_n)$]
\label{lem:gd_psd_general}
Let $\operatorname{rank}(S_n)=r\ge 1$, and define $L_n := \lambda_{\max}(S_n)$ and
$\mu_n = \lambda_{\min}\!\big(S\big|_{\operatorname{range}(S_n)}\big)$ (i.e., the minimum non-zero eigenvalue of $S_n$). Consider gradient descent with step size
$\eta=1/L_n$, initialized at $\theta^{(0)}\in\operatorname{range}(S_n)$.
Then, for all $t\ge 0$,
\[
g_n(\theta^{(t)})-g_n(\theta^\star)
\le
q^{\,2t}\,\big\{g_n(\theta^{(0)})-g_n(\theta^\star)\big\}
\le
q^{\,t}\,\big\{g_n(\theta^{(0)})-g_n(\theta^\star)\big\},
\]
where $q_n=\mu_n/L_n$.
\end{lem}

\noindent We note that the assumption $\theta^{(0)} \in \operatorname{range}(S_n)$ is not restrictive.
Indeed, the zero initialization $\theta^{(0)}=0$ always satisfies this condition.
Moreover, when $\operatorname{rank}(S_n)=p$ (in particular, in the classical case
$p\le n$ where $S_n$ is nonsingular), we have
$\operatorname{range}(S_n)=\mathbb{R}^p$, and no restriction on the initialization
is required. In the rank-deficient case, one may either initialize at
$\theta^{(0)}=0$ or project an arbitrary initial point onto
$\operatorname{range}(S_n)$, both of which ensure convergence of gradient descent
to the Moore--Penrose solution $\theta_n^\star=S_n^{+}b_n$.

\medskip

\noindent
{\bf Analyzing iteration complexity of gradient descent} In practical applications, it is important to
understand how many iterations are required to reach a prescribed accuracy. To
quantify this, the iteration complexity $T_{\varepsilon,n}$ (see, e.g.,
\cite{bubeck2015convex}) is defined by
\[
T_{\varepsilon,n}
:=
\inf\Big\{t\ge 0:\ g_n(\theta^{(t)})-g_n(\theta_n^\star)\le
\varepsilon\big(g_n(\theta^{(0)})-g_n(\theta_n^\star)\big)\Big\}.
\]

\noindent We will now show that, under the Gaussianity assumption and for sufficiently large sample size $n$, and under an
appropriate choice of step size $\eta$, the expected iteration complexity of
gradient descent for solving the Gram system \eqref{eq:gram_system} remains
bounded by a constant independent of $n$ whenever $p/n\to\gamma\in[0,\infty)\setminus\{1\}$. The key observation is that,
even though $g_n$ may fail to be strongly convex on $\mathbb{R}^p$ when $S_n$ is
rank-deficient, Lemma~\ref{lem:gd_psd_general} guarantees linear convergence of
the iterates on $\operatorname{range}(S_n)$, provided that
$\theta^{(0)}\in\operatorname{range}(S_n)$ and $\eta$ is chosen appropriately.
As a result, the iteration complexity is governed by the square of the condition
number of $X_n$. In particular, controlling the second moment of $\kappa(X_n)$
is central to obtaining uniform bounds on the expected runtime.
\begin{thm}[Mean iteration bound for gradient descent on random Gram systems]
\label{thm:gd_gram_condition_number}
Assume that the rows of $X_n\in\mathbb{R}^{n\times p}$ are independent and
identically distributed as $N(0,\tilde\Sigma_n)$, where $\tilde\Sigma_n$
satisfies \eqref{cond-const}. Fix $\varepsilon\in(0,1)$ and consider the gradient descent
iterates $\{\theta^{(t)}\}_{t\ge 0}$ with step size $\eta=\frac{1}{\lambda_{\max}(S_n)}$ and initialized at $\theta^{(0)}\in\operatorname{range}(S_n)$. Then the iteration complexity $T_{\varepsilon,n}$ satisfies, for every $n$,
\[
T_{\varepsilon,n}
\le
\kappa(X_n)^2\,\log\!\left(\frac{1}{\varepsilon}\right)+1.
\]
Moreover, if $p/n\to\gamma\in[0,\infty)\setminus \{1\}$ as
$n\to\infty$, there exists a constant $C>0$ (depending only on $c_m,c_M$ and
$\gamma$) such that for all sufficiently large $n$,
\[
\mathbb{E}\big[T_{\varepsilon,n}\big]
\le
C\log\!\left(\frac{1}{\varepsilon}\right)+1.
\]
\end{thm}

\begin{proof} First, note that \eqref{cond-const} implies $\tilde\Sigma_n\succ 0$ for all $n$,
and therefore $\mathrm{rank}(\tilde X_n)=\min(n,p)$ almost surely. In particular,
in the wide-matrix regime $p>n$, the smallest nonzero eigenvalue of $S_n$ equals the square of the smallest singular value of $X_n$ (equivalently, the $n$-th singular value).

Define $L_n$ and $\mu_n$ as in Lemma \ref{lem:gd_psd_general}. Applying the results as in  Lemma \ref{lem:gd_psd_general} with $r= \min(n,p)$, it follows that for every $t\ge0$,
\[
g_n(\theta^{(t)})-g_n(\theta_n^\star)
\le
\left(1-\frac{\mu_n}{L_n}\right)^t
\Big\{g_n(\theta^{(0)})-g_n(\theta_n^\star)\Big\}.
\]
Let $q_n:=1-\mu_n/L_n\in(0,1)$. Define, $T^*_{\varepsilon,n}:=\inf\Big\{t\ge 0:\ q_n^{\,t}\le \varepsilon\Big\}$. Then by definition, $T^*_{\varepsilon,n} \geq T_{\varepsilon,n}$.

If we exhibit a (deterministic) integer $t_0$ such that $q_n^{t_0}\le\varepsilon$,
then $t_0$ belongs to the set inside the infimum and therefore
$T^*_{\varepsilon,n}\le t_0$.
Using $\log(1-x)\le -x$ for $x\in(0,1)$ with $x=\mu_n/L_n$ gives
\[
q_n^{\,t}=\exp\!\big(t\log q_n\big)\le \exp\!\left(-t\frac{\mu_n}{L_n}\right).
\]
Thus, choosing
\[
t_0:=\left\lceil \frac{L_n}{\mu_n}\log\!\left(\frac{1}{\varepsilon}\right)\right\rceil
\]
yields $\exp\!\left(-t_0\frac{\mu_n}{L_n}\right)\le \varepsilon$ and hence
$q_n^{t_0}\le \varepsilon$. Consequently,
\[
T_{\varepsilon,n} \le T^*_{\varepsilon,n}
\le
t_0
\le
\frac{L_n}{\mu_n}\log\!\left(\frac{1}{\varepsilon}\right)+1 = \kappa(X_n)^2 \log\!\left(\frac{1}{\varepsilon}\right)+1.
\]
For the expectation bound, note that
\[
\mathbb{E}\big[T_{\varepsilon,n}\big]
\le
\log\!\left(\frac{1}{\varepsilon}\right)\mathbb{E}\big[\kappa(X_n)^2\big]+1.
\]
By Theorem~\ref{theorem_condition_number} and Lemma \ref{lem:double_descent} with $r=2$, under $p/n\to\gamma\in[0,\infty)\setminus \{1\}$
and \eqref{cond-const}, there exists a constant $C>0$ (depending only on
$c_m,c_M$ and $\gamma$) such that $\sup_n \mathbb{E}[\kappa(X_n)^2]\le C$.
This completes the proof.
\end{proof}

\begin{remark}\label{rem:krylov}
Note that the gradient descent iterates belong to the Krylov subspaces
\[
\mathcal{K}_t(S_n,b)=\mathrm{span}\{b,S_nb,\dots,S_n^{t-1}b\},
\]
since each iterate can be written as a polynomial in $S_n$ applied to $b$; see,
for example, \cite{saad2003iterative,trefethen1997numerical}. More generally,
Krylov subspace methods (e.g., the conjugate gradient algorithm) construct
iterates by optimizing over these subspaces, often yielding faster convergence
for symmetric positive definite systems such as \eqref{eq:gram_system}; see, for
example, the discussion after \cite[Theorem~38.5]{trefethen1997numerical}. In
particular, such methods ensure $T_{\varepsilon,n}=O(\kappa(X_n))$. However, even in this case Theorem~\ref{theorem_condition_number} and Lemma \ref{lem:double_descent}  (with $r=1$) will ensure that $\mathbb{E}\big[T_{\varepsilon,n}\big]$ is finite when $p/n\to\gamma\in[0,\infty)\setminus \{1\}$.
\end{remark}

\noindent
Our next result shows that the worst-case iteration complexity, denoted 
by $T_{\varepsilon,n}^{\mathrm{wc}}$, admits a lower bound proportional to $(L_n-\mu_n)/\mu_n$. In the critical regime $p/n\to 1$, this leads to a sharp computational instability, since $\mathbb{E}[T_{\varepsilon,n}^{\mathrm{wc}}]\to\infty$
for every fixed $\varepsilon\in(0,1)$. Hence, stability of gradient-based
iterative procedures in high dimensions is fundamentally tied to moment control
of the condition number of the design matrix. This phenomenon is closely related
to the observations of \cite{poggio2020doubledescentconditionnumber}, who
highlighted a ``double descent'' behavior of the condition number in the context
of solving linear systems. In particular, from the viewpoint of computational
complexity, it is typically easier to solve an under-determined ($n \gtrsim p$) or
over-determined ($p \gtrsim n$) linear system than a nearly well-determined
system with $n \approx p$. 
\begin{lem}\label{lem:GD_worstcase}
Assume that the rows of $X_n\in\mathbb{R}^{n\times p}$ are independent and
identically distributed as $N(0,\tilde\Sigma_n)$, where $\tilde\Sigma_n$
satisfies \eqref{cond-const}. Fix $\varepsilon\in(0,1)$, and define the worst case version of $T_{\varepsilon,n}$
\[
T_{\varepsilon,n}^{\mathrm{wc}}
\ :=\ 
\underset{\substack{\theta^{(0)}\neq \theta_n^\star}}{\sup}T_{\varepsilon,n}.
\]
Then,
\begin{equation}\label{eq:Twc_kappa}
T_{\varepsilon,n}^{\mathrm{wc}}
\ \ge\
\left\lceil
\frac{L_n-\mu_n}{2\mu_n}\,\log\!\Big(\frac{1}{\varepsilon}\Big)
\right\rceil .
\end{equation}

\noindent Finally, if
$p/n\to \gamma = 1$ for $n \to \infty$, then, for every fixed $\varepsilon\in(0,1)$,
\begin{equation}\label{eq:ETwc_lower}
\mathbb{E}\big[T_{\varepsilon,n}^{\mathrm{wc}}\big]\to \infty,
\end{equation}
as $n \to \infty$.
\end{lem}

\begin{proof}
Let $e^{(t)}:=\theta^{(t)}-\theta_n^\star$. Since $\theta_n^\star=S_n^{-1}b$, we have
$\nabla g_n(\theta)=S_n(\theta-\theta_n^\star)=S_n e$, and hence
\[
e^{(t+1)}=e^{(t)}-\frac{1}{L_n}S_n e^{(t)}
=\Big(I-\frac{1}{L_n}S_n\Big)e^{(t)}.
\]
Let $v_{\min}$ be a unit eigenvector of $S_n$ associated with $\mu_n$, and consider an initialization
such that $e^{(0)}=\alpha v_{\min}$ for some $\alpha\neq 0$. Then, using the recursion in \eqref{eq_gd_linear1}, it follows that
\[
e^{(t)}=\alpha\Big(1-\frac{\mu_n}{L_n}\Big)^t v_{\min}.
\]
Moreover,
\[
g_n(\theta^{(t)})-g_n(\theta_n^\star)
=\frac12\,(e^{(t)})^\top S_n e^{(t)}
=\frac12\,\mu_n\alpha^2\Big(1-\frac{\mu_n}{L_n}\Big)^{2t},
\]
and therefore
\[
\frac{g_n(\theta^{(t)})-g_n(\theta_n^\star)}{g_n(\theta^{(0)})-g_n(\theta_n^\star)}
=\Big(1-\frac{\mu_n}{L_n}\Big)^{2t}.
\]
Consequently, the condition
$g_n(\theta^{(t)})-g_n(\theta_n^\star)\le \varepsilon\,(g_n(\theta^{(0)})-g_n(\theta_n^\star))$
holds if and only if (for this particular initialization),
$T_{\varepsilon,n} \ge \left\lceil \frac{\log(1/\varepsilon)}{-2\log(1-\mu_n/L_n)}\right\rceil$.
Taking the supremum over $\theta^{(0)}\neq\theta_n^\star$ yields 
\begin{align*}
  T_{\varepsilon,n}^{\mathrm{wc}}  \geq  \left\lceil \frac{\log(1/\varepsilon)}{-2\log(1-\mu_n/L_n)}\right\rceil
\end{align*}

For \eqref{eq:Twc_kappa}, set $x=\mu_n/L_n\in(0,1)$ and use the elementary bound $-\log(1-x)\le x/(1-x)$ for $x\in(0,1)$ to obtain
\[
T_{\varepsilon,n}^{\mathrm{wc}}
\ \ge\
\left\lceil
\frac{L_n-\mu_n}{2\mu_n}\,\log\!\Big(\frac{1}{\varepsilon}\Big)
\right\rceil,
\]
which yields \eqref{eq:Twc_kappa} after taking ceilings.

Finally, if
$p/n\to \gamma = 1$ for $n \to \infty$ Then, for every fixed $\varepsilon\in(0,1)$, applying Lemma \ref{lem:double_descent} (with $r =2$, since $L_n/\mu_n=\kappa(X_n)^2$) gives $\mathbb{E}\big[T_{\varepsilon,n}^{\mathrm{wc}}\big]\to \infty$ as $n \to \infty$.
\end{proof}

\section{Sub-Gaussian Extensions: Limitations and Caveats}\label{sec_subg}

\noindent
In the high-dimensional statistics literature, many results established under Gaussian
assumptions continue to hold under sub-Gaussianity. This is 
largely because, by definition, a sub-Gaussian random variable has tails that are
dominated by those of an appropriate Gaussian distribution. As a consequence,
many asymptotic and nonasymptotic analyses that rely primarily on tail
probability bounds extend naturally from the Gaussian to the sub-Gaussian
setting. In particular, as noted in Remark \ref{remark:subg}, the moment bounds for the maximum singular value in Lemma \ref{max_singular_value} can be extended to general sub-Gaussian settings. However, we show in this section that this is {\bf not} the case for the smallest singular value. In particular, our main results in Section \ref{sec:main_results}---Theorem~\ref{theorem_condition_number} and Theorem~\ref{min_singular_value}---fail to hold in general sub-Gaussian settings. We start by recalling the definition of a sub-Gaussian distribution.

\newcommand{\snorm}[1]{\left\lVert#1\right\rVert_{\psi_2}}

\begin{definition}[\textbf{Sub-Gaussian random variable}]\label{def2}
A mean-zero random variable $X$ is called \emph{sub-Gaussian} if there exists a
constant $k_1>0$ such that
\[
\mathbb{P}\big(|X|>t\big)\le 2\exp\!\left(-\frac{t^2}{k_1^2}\right),
\qquad \text{for all } t\ge 0 .
\]
\end{definition}

\noindent
If $X$ is a sub-Gaussian random variable, then it satisfies
\[
\big(\mathbb{E}|X|^p\big)^{1/p}\leq k_2\sqrt{p}
\]
for some constant $k_2>0$. The sub-Gaussian norm of $X$ is defined by
\[
\snorm{X}\coloneqq \sup_{\substack{p\geq1}}\;p^{-1/2}\big(\mathbb{E}|X|^p\big)^{1/p}.
\]

\noindent
A mean-zero random vector $X\in\mathbb{R}^q$ is said to be \emph{sub-Gaussian} if
for every $u\in\mathcal{S}^{q-1}$, the random variable $u^{\top}X$ is
sub-Gaussian. The sub-Gaussian norm of the random vector $X$ is defined as
\[
\snorm{X}\coloneqq \sup_{\substack{u\in\mathcal{S}^{q-1}}}\snorm{u^{\top}X}.
\]

\noindent
See \cite{vershynin2018high} for further details. We now provide the promised counterexample. 
\begin{prop}\label{prop:subg_counterexample_r1}
Let $X_n \in \mathbb{R}^{n \times p}$ have i.i.d.\ entries $\{X_{ij}\}$, where $X_{ij} = Z_{ij} U_{ij}$ for all $(i,j)$. The variables $U_{ij}$ (supported on (0,1)) are i.i.d. with distribution function
\begin{equation} \label{sub:Gaussian:CDF}
F(u) = \frac{1}{\log\!\big(\frac{e}{u}\big)}, \qquad u \in (0,1).
\end{equation}

The variables $Z_{ij}$ are i.i.d.\ Rademacher random variables independent of $U_{ij}$; that is, $P(Z_{ij}=1)=P(Z_{ij}=-1)=\frac{1}{2}$. Then $X_{ij}$ is mean-zero sub-Gaussian (hence $X_n$ has i.i.d.\ mean-zero sub-Gaussian entries). Then, for every $n \geq 1$ and $p \geq 2$,
\[
\mathbb{E}\!\left[\frac{\sqrt{n}}{s_{\min}(X_n)}\right] = \infty,
\qquad
\mathbb{E}\big[\kappa(X_n)\big] = \infty.
\]
\end{prop}

\begin{proof}
Since $-1 <X_{ij}<1$ almost surely, the random variable $X_{ij}$ is bounded and
therefore sub-Gaussian. Also, $\mathbb{E}[X_{ij}]= \mathbb{E}[ Z_{ij} U_{ij}]=  \mathbb{E}[ Z_{ij}]\mathbb{E}[U_{ij}]=0$.

Fix $\epsilon\in(0,1)$ and any column index $j\in\{1,\dots,p\}$. Define
\[
E_j(\epsilon)=\{|X_{ij}|\le \epsilon \text{ for all } i=1,\dots,n\}.
\]
By independence and since $|X_{ij}| \overset{d}{=} U_{ij}$,
\[
\mathbb{P}\big(E_j(\epsilon)\big)=\big(F(\epsilon)\big)^n
=\left(\frac{1}{\log\!\big(\frac{e}{\epsilon}\big)}\right)^n.
\]
On the event $E_j(\epsilon)$ we have
\[
s_{\min}(X_n)\le \|X_n e_j\|_2 \le \epsilon\sqrt{n},
\]
and hence
\begin{align}\label{eq_counter_example}
    \mathbb{P}\!\left(\frac{\sqrt{n}}{s_{\min}(X_n)}\ge \frac{1}{\epsilon}\right)
\ge
\left(\frac{1}{\log\!\big(\frac{e}{\epsilon}\big)}\right)^n.
\end{align}
Using the tail integral representation $\mathbb{E}[Y]=\int_0^\infty
\mathbb{P}(Y>t)\,dt$, we obtain
\begin{align*}
\mathbb{E}\!\left[\frac{\sqrt{n}}{s_{\min}(X_n)}\right]
&\ge
\int_{1}^{\infty}
\mathbb{P}\!\left(\frac{\sqrt{n}}{s_{\min}(X_n)}\ge t\right)\,dt
\ge
\int_{1}^{\infty}
\left(\frac{1}{\log(et)}\right)^n dt.
\end{align*}
The last integral is infinite for every fixed $n\ge 1$ since
$\int_1^\infty (\log t)^{-n}\,dt = \infty$  for every fixed $n\ge 1$. This proves
$\mathbb{E}[\sqrt{n}/s_{\min}(X_n)]= \infty$  for every fixed $n\ge 1$.

To show $\mathbb{E}[\kappa(X_n)] = \infty$, drop column $j$ from matrix $X_n$ and denote the
remaining matrix by $X_{-j}$. Since the entries are i.i.d. and hence independent
across columns, the event $E_j(\epsilon)$ is independent of any event determined
by $X_{-j}$. Let $m_2=\mathbb{E}[X_{ij}^2]\in(0,1)$ and define
\[
G_j=\left\{\|X_{-j}\|_F^2 \ge \frac{1}{2}n(p-1)m_2\right\}.
\]
Because $0\le X_{ij}^2\le 1$, a Bernstein-type inequality (\cite{vershynin2011introductionnonasymptoticanalysisrandom}) yields
$\mathbb{P}(G_j)\ge 1-\exp\{- cn(p-1)\}$ for some universal constant $c>0$. Set
$Q_0(n):=\mathbb{P}(G_j)$ and note that $Q_0(n) > 0$ for every $n \geq 1$ and $p \geq 2$.

On $G_j$,
\[
s_{\max}(X_n)\ge \frac{\|X_n\|_F}{\sqrt{p}}
\ge \frac{\|X_{-j}\|_F}{\sqrt{p}}
\ge \left(\sqrt{\frac{m_2}{4}\frac{p-1}{p}}\right)\sqrt{n}
=:C^*(p)\sqrt{n}.
\]
Therefore, on $E_j(\epsilon)\cap G_j$,
\[
\kappa(X_n)=\frac{s_{\max}(X_n)}{s_{\min}(X_n)}
\ge
\frac{C^*(p)\sqrt{n}}{\epsilon\sqrt{n}}
=
\frac{C^*(p)}{\epsilon}.
\]
By independence of $E_j(\epsilon)$ and $G_j$,
\[
\mathbb{P}\!\left(\kappa(X_n)\ge \frac{C^*(p)}{\epsilon}\right)
\ge
\mathbb{P}(E_j(\epsilon))\,\mathbb{P}(G_j)
=
Q_0(n)\left(\frac{1}{\log\!\big(\frac{e}{\epsilon}\big)}\right)^n.
\]
Let $t=C^*(p)/\epsilon$ (so $\epsilon=C^*(p)/t$). Then for all $t\ge C^*(p)$,
\[
\mathbb{P}\big(\kappa(X_n)\ge t\big)
\ge
Q_0(n)\left(\frac{1}{\log\!\big(\frac{et}{C^*(p)}\big)}\right)^n.
\]
Using $\mathbb{E}[Y]=\int_0^\infty \mathbb{P}(Y>t)\,dt$ for $Y\ge 0$ yields
\[
\mathbb{E}[\kappa(X_n)]
\ge
\int_{C^*(p)}^\infty \mathbb{P}(\kappa(X_n)\ge t)\,dt
\ge Q_0(n) \int_{\sqrt{\frac{m_2}{4}}}^\infty
\left(\frac{1}{\log\!\big(\frac{8et}{\sqrt{m_2}}\big)}\right)^n dt,
\]
for $n \geq 1$. The last step follows since $(p-1)/p \in [1/2,1)$ for every $p \ge 2$. Hence, the last integral diverges since for all $n \geq 1$ and $p \geq 2$, $Q_0(n) > 0$, implying $\mathbb{E}[\kappa(X_n)]= \infty$  for every fixed $n\ge 1$ and $p \geq 2$. This completes the proof.
\end{proof}

\begin{remark}
Proposition~\ref{prop:subg_counterexample_r1} immediately implies that for every
$r>1$ and $n \ge 1$ and $p \ge 2$,
\[
\mathbb{E}\!\left[\left(\frac{\sqrt{n}}{s_{\min}(X_n)}\right)^r\right] = \infty,
\qquad
\mathbb{E}\!\left[\kappa(X_n)^r\right] = \infty,
\]
by the standard monotonicity property of moments. For $0<r<1$, the same
tail-integral argument used in the proof of
Proposition~\ref{prop:subg_counterexample_r1} yields the same conclusion; since
the steps are virtually identical, we omit the details. Hence, the conclusions of
Theorem~\ref{min_singular_value} and Theorem~\ref{theorem_condition_number} do not extend to general sub-Gaussian designs.
\end{remark}

\begin{remark}
 Proposition~\ref{prop:subg_counterexample_r1} may be viewed as a \emph{random matrix} example of a phenomenon well known in the general theory of convergence of random variables. While \cite[Theorem~5.39]{vershynin2011introductionnonasymptoticanalysisrandom} establishes high-probability bounds for the minimum singular value under broad sub-Gaussian assumptions, these bounds govern convergence in probability. Thus, Proposition~\ref{prop:subg_counterexample_r1} may be viewed as a random matrix instance where convergence in probability holds but convergence in expectation (and hence moment stability) fails.  
\end{remark}

\noindent
We next show that for the sub-Gaussian distribution in
Proposition~\ref{prop:subg_counterexample_r1}, the analogue of Theorem \ref{thm:inv_cov_error_rate} does not hold. In particular, let $X_n \in \mathbb{R}^{n \times p}$ have i.i.d.\ entries $\{X_{ij}\}$, where $X_{ij} = Z_{ij} U_{ij}$ for all $(i,j)$, where $\{U_{ij}\}$ and $\{Z_{ij}\}$ are distributed as in the statement of Proposition \ref{prop:subg_counterexample_r1}. Define the (uncentered) sample covariance (Gram) matrix
\[
S_n:=\frac{1}{n}X_n^\top X_n,
\]
and let $\Sigma:=\mathbb{E}[S_n]$. We have already discussed in Section~\ref{sec:cov} that the positive moments of the estimation error of $S_n$, namely $\mathbb{E}\big[\|S_n - \Sigma\|_2^r\big]$, admit a nonasymptotic bound in the Gaussian case (see \eqref{moment_bound_S}), which can be extended to the sub-Gaussian setting (\ \cite{gram}).

Our next result shows, however, that the expected estimation error of the inverse sample covariance, namely $\mathbb{E}\big[\|S_n^{-1} - \Sigma^{-1}\|_2\big]$, may diverge to $\infty$ for a particular class of sub-Gaussian distributions. Consequently, in contrast to $\mathbb{E}\big[\|S_n - \Sigma\|_2\big]$, it is not possible to establish a uniform rate for $\mathbb{E}\big[\|S_n^{-1} - \Sigma^{-1}\|_2\big]$ over a general class of sub-Gaussian distributions.

\begin{prop} \label{prop:inv_cov_infinite_mean}
Let $X_n \in \mathbb{R}^{n \times p}$ have i.i.d. entries with distribution function given in (\ref{sub:Gaussian:CDF}), and let $S_n = \frac{1}{n} X_n^\top X_n$ denote the corresponding sample covariance matrix. Then 
\[
\mathbb{E}\big[\|S_n^{-1}-\Sigma^{-1}\|_2\big] = \infty,
\]
for every $n \geq 1$ and $p \leq n$.
\end{prop}

\begin{proof}
Observe that for $p \leq n$
\[
\|S_n^{-1}\|_2=\frac{1}{\lambda_{\min}(S_n)}
=\frac{1}{\lambda_{\min}\!\left(\frac{1}{n}X_n^\top X_n\right)}
=\frac{n}{s_{\min}(X_n)^2}
=\left(\frac{\sqrt{n}}{s_{\min}(X_n)}\right)^2.
\]
By Proposition~\ref{prop:subg_counterexample_r1} (and the same tail-integral
argument used there), we have
\[
\mathbb{E}\!\left[\left(\frac{\sqrt{n}}{s_{\min}(X_n)}\right)^2\right]= \infty,
\]
for every $n \geq 1$ and $p \leq n$. and therefore $\mathbb{E}\|S_n^{-1}\|_2 = \infty$ for every $n \geq 1$..

Finally, using the reverse triangle inequality,
\[
\|S_n^{-1}-\Sigma^{-1}\|_2 \ge \|S_n^{-1}\|_2-\|\Sigma^{-1}\|_2.
\]
Taking expectations gives
\[
\mathbb{E}\big[\|S_n^{-1}-\Sigma^{-1}\|_2\big]
\ge
\mathbb{E}\|S_n^{-1}\|_2-\|\Sigma^{-1}\|_2
= \infty,
\]
for every $n \geq 1$ and $p \leq n$.
\end{proof}

\subsection{Connection to small-ball probabilities for the smallest singular value}\label{connection:small_ball}

To understand why our results, which hold in the Gaussian case, do not extend to
a general class of sub-Gaussian distributions, it is useful to recall a
well-known concept in random matrix theory known as \textit{small-ball
probabilities}; see, e.g.,
\cite{littlewood_offword,Li1999MetricEntropySmallBall,
Tikhomirov2016SmallestSingularNoMoments}. Roughly speaking, the term
\textit{small-ball probability} refers to lower bounds of the form
\[
\mathbb{P}\big(|\langle X_n,u\rangle|\ge \epsilon\big)\ge c
\qquad \text{for all } u\in\mathcal{S}^{p-1},
\]
for some fixed constants $\epsilon,c>0$. Intuitively, it quantifies how
unlikely it is for a random variable (or a linear functional of a random
vector) to fall inside a small neighborhood of zero. Such bounds play a
fundamental role in controlling the smallest singular value of random matrices,
since excessive mass near zero may allow certain columns or directions to
become nearly degenerate and thus force $s_{\min}(X_n)$ to be extremely small.

To connect our results to the existing literature, we recall
\cite[Theorem~1.1]{RudelsonVershynin2009SmallestSingularRectangular}, stated
below in our notation. 

\begin{thm}[\cite{RudelsonVershynin2009SmallestSingularRectangular}, Theorem 1.1] \label{thm_vershynin}
Let $Z_n\in\mathbb{R}^{n\times p}$ with $n\ge p$ have i.i.d.\ entries which are
independent copies of a centered sub-Gaussian random variable with unit
variance. Then there exist constants $C,c>0$, depending only (polynomially) on
the sub-Gaussian norm $\|Z_{11}\|_{\psi_2}$, such that for every $\varepsilon>0$,
\[
\mathbb{P}\!\left(
s_{\min}(Z_n)\le \varepsilon\big(\sqrt{n}-\sqrt{p-1}\big)
\right)
\le
(C\varepsilon)^{\,n-p+1}+\exp\{-cn\}.
\]
\end{thm}
The above result shows that for general i.i.d.\ sub-Gaussian designs,
the lower tail of $s_{\min}(Z_n)$ admits at best a polynomial decay in
$\varepsilon$, up to an additional exponential term $\exp\{-cn\}$ which does not
depend on $\varepsilon$. In the Gaussian case, our Theorem~\ref{min_singular_value}
(see in particular \eqref{lemma min:5}) yields a polynomial-type bound of the
same flavor, {\em but importantly does not require the extra exponential term}. 

This distinction is crucial. Indeed, as emphasized in
\cite{RudelsonVershynin2009SmallestSingularRectangular}, the presence of the
term $\exp\{-cn\}$ is unavoidable for general sub-Gaussian settings. To see this, consider the case
where $Z_n$ has i.i.d.\ Rademacher entries, taking values $\pm 1$ with equal
probability. Then the first two columns of $Z_n$ coincide with probability
\[
\mathbb{P}\big(Z_{\cdot 1}=Z_{\cdot 2}\big)
=
2^{-n+1},
\]
since, for each of the $n$ rows, the two corresponding entries agree with
probability $1/2$, and the rows are independent. On this event, the matrix has
two identical columns and is therefore rank-deficient, which implies
$s_{\min}(Z_n)=0$. Hence,
\[
\mathbb{P}\big(s_{\min}(Z_n)=0\big)\ge 2^{-n+1}=\exp\{-c'n\}
\]
for a universal constant $c'>0$. This shows that an additive term of order
$\exp\{-cn\}$ in Theorem~\ref{thm_vershynin} cannot, in general, be removed by
taking $\varepsilon$ smaller (or even by taking $\varepsilon=0$), and is thus
optimal up to constants. 

It is also important to note that the exponential term in Theorem~\ref{thm_vershynin} is not merely a
technical artifact required to accommodate discrete distributions. In fact, in Proposition~\ref{prop:subg_counterexample_r1}, we construct a continuous sub-Gaussian distribution for which the lower tail of $s_{\min}(Z_n)$ near the origin decays substantially more slowly, at a logarithmic--polynomial rate - leading to the explosion of moments of 
$\frac{\sqrt{n}}{s_{\min}(Z_n)}$ as $n$ increases. This demonstrates that even within the class of continuous sub-Gaussian
distributions, it is not guaranteed that the small-ball behavior of
$s_{\min}(Z_n)$ exhibits a purely polynomial decay in $\varepsilon$.

One of our key contributions is to show that under Gaussianity, this additional term can be avoided. This is essential for our main objective of obtaining finite moments for
inverse quantities such as $\big(\sqrt{n}/s_{\min}(Z_n)\big)^r$ and $\kappa(Z_n)^r$.

Intuitively, the divergence in Proposition~\ref{prop:subg_counterexample_r1} is a consequence of the fact that 
the entries have an unbounded density near the origin, placing substantial mass
close to zero. This increases the probability that certain columns become
abnormally small, forcing the smallest singular value to be extremely small and leading to
divergence of inverse moments and condition number moments. 
This leads to another interesting question: can our results, namely
Theorem~\ref{min_singular_value} and Theorem~\ref{theorem_condition_number}, be
extended to a broad class of sub-Gaussian distributions whose densities are
bounded in a neighborhood of the origin? While this appears plausible at an
intuitive level, it is not immediate from the existing literature, nor from the
structure of our proof of Theorem~\ref{min_singular_value}.

To see the source of the difficulty, recall the following key step in the proof of
Theorem~\ref{min_singular_value} (see \eqref{lemma min:1}--\eqref{lemma min:4}).
Let $W_n=X_n^\top X_n$, where $X_n$ has i.i.d.\ $\mathcal{N}(0,1)$ entries. A crucial
Gaussian feature is that for any fixed $u\in\mathcal{S}^{p-1}$,
\[
u^\top W_n^{-1}u \sim \text{Inverse-}\chi^2_{n-p+1}.
\]
This identity allows us to convert the problem of bounding $s_{\min}(X_n)$ into a
problem of controlling $\lambda_{\max}(W_n^{-1})$, since
$\lambda_{\max}(W_n^{-1})=\sup_{\|u\|_2\le 1}u^\top W_n^{-1}u$. The latter is
amenable to a standard covering argument: the supremum over the unit ball can
be reduced to a supremum over a finite $\varepsilon$-net, and any tail bound for
a fixed quadratic form $u^\top W_n^{-1}u$ can be upgraded to a uniform bound
over the net via a union bound. In other words, once one has strong control for
each fixed direction $u$, the global supremum can be controlled at essentially
the same scale, up to the net cardinality.

For a general sub-Gaussian design, even if the common distribution of the entries has a density
bounded near zero, the exact inverse-$\chi^2$ distribution of $u^\top W_n^{-1}u$
is no longer available. As a result, we cannot leverage the same direct route to
a sharp control of $\lambda_{\max}(W_n^{-1})$. Instead, existing approaches for
sub-Gaussian matrices typically proceed by constructing an appropriate
\textit{good set} on which one can control the relevant quadratic forms
uniformly (see, e.g., \cite[Lemma~7.4 and Lemma~7.6]{RudelsonVershynin2009SmallestSingularRectangular}).
The contribution of the complement of this good set is then controlled by a
separate argument, and it is precisely at this stage that an additional term of
the form $\exp\{-cn\}$ appears. Consequently, even for sub-Gaussian distributions
with densities bounded near the origin, it is not clear how one could avoid the
extra $\exp\{-cn\}$ term using the above good set approach, and this is
exactly the obstruction to obtaining finite inverse moments for the smallest
singular value and for the condition number in the same way as in the Gaussian
case. 

That said, extending our inverse-moment and condition-number results (Theorem \ref{min_singular_value} and Theorem \ref{theorem_condition_number}) to a
natural subclass of sub-Gaussian distributions with bounded densities near the
origin remains an interesting and nontrivial problem. We view this as a
promising direction for future research building on the present work.

\appendix

\section{The first appendix}\label{apnd_A}

\subsection{Proof of Lemma \ref{lem:multi_variance_term}}

\noindent Let $S_n=(1/n)X_n^\top X_n$ and set $A=(S_n+\tilde{\lambda}_n I_p)^{-1}$. Under (\ref{model}),
$Y_n=X_nB+E$, so
\[
\hat B_\lambda
=
A\frac{1}{n}X_n^\top Y_n
=
A\Big(S_nB+\frac{1}{n}X_n^\top E\Big),
\qquad
\hat B_\lambda-B = -\tilde{\lambda}_n AB + A\frac{1}{n}X_n^\top E.
\]
Hence
\[
X_n(\hat B_\lambda-B)= -\tilde{\lambda}_n X_nAB + X_nA\frac{1}{n}X_n^\top E.
\]
Conditioning on $X_n$, the cross term vanishes since $\mathbb{E}(E\mid X_n)=0$,
and therefore
\begin{align*}
R_{\mathrm{pred}}(\tilde{\lambda}_n\mid X_n)
&=
\frac{1}{n}\mathbb{E}\big[\|X_n(\hat B_\lambda-B)\|_F^2\mid X_n\big]\\
&=
\frac{\tilde{\lambda}_n^2}{n}\|X_nAB\|_F^2
+\frac{1}{n}\mathbb{E}\Big[\Big\|X_nA\frac{1}{n}X_n^\top E\Big\|_F^2\Bigm|X_n\Big].
\end{align*}
For the first term,
\[
\frac{1}{n}\|X_nAB\|_F^2
=
\operatorname{tr}\!\Big(B^\top A\Big(\frac{1}{n}X_n^\top X_n\Big)AB\Big)
=
\operatorname{tr}\!\Big(B^\top AS_nAB\Big),
\]
which yields the stated bias term.

For the second term, define $M=(1/n)X_nAX_n^\top$. Then
\[
X_nA\frac{1}{n}X_n^\top E = ME.
\]
Write the rows of $E$ as $E_1^\top,\ldots,E_n^\top$, with $E_i\sim N(0,\Sigma_\varepsilon)$
i.i.d. For each $i$,
\[
(ME)_i^\top=\sum_{j=1}^n M_{ij}E_j^\top,
\]
and independence of the rows implies
\[
\mathbb{E}\big[\|(ME)_i\|_2^2\mid X_n\big]
=
\sum_{j=1}^n M_{ij}^2\,\mathbb{E}\big[\|E_j\|_2^2\big]
=
\mathrm{tr}(\Sigma_\varepsilon)\sum_{j=1}^n M_{ij}^2.
\]
Summing over $i$ gives
\[
\mathbb{E}\big[\|ME\|_F^2\mid X_n\big]
=
\mathrm{tr}(\Sigma_\varepsilon)\sum_{i,j} M_{ij}^2
=
\mathrm{tr}(\Sigma_\varepsilon)\,\operatorname{tr}(M^2).
\]
Using $M=(1/n)X_nAX_n^\top$ and cyclicity of trace,
\[
\operatorname{tr}(M^2)
=
\operatorname{tr}\!\left(\frac{1}{n}X_nAX_n^\top\frac{1}{n}X_nAX_n^\top\right)
=
\operatorname{tr}\!\left(S_nAS_nA\right)
=
\operatorname{tr}\!\Big(S_n(S_n+\tilde{\lambda}_n I_p)^{-1}S_n(S_n+\tilde{\lambda}_n I_p)^{-1}\Big),
\]
which proves the stated variance term.

To bound the bias term, diagonalize $S_n=U\Lambda U^\top$ with
$\Lambda=\mathrm{diag}(\lambda_1,\ldots,\lambda_p)$. Then
\[
AS_nA
=
U\,\mathrm{diag}\!\left(\frac{\lambda_1}{(\lambda_1+\tilde{\lambda}_n)^2},\ldots,
\frac{\lambda_p}{(\lambda_p+\tilde{\lambda}_n)^2}\right)U^\top,
\]
so
\[
\operatorname{tr}\!\Big(B^\top AS_nAB\Big)
=
\operatorname{tr}\!\Big(BB^\top AS_nA\Big)
\le
\|BB^\top\|_2\,\operatorname{tr}(AS_nA)
\le
\|B\|_F^2\,\lambda_{\max}(AS_nA).
\]
Moreover, for each $j$,
\[
\frac{\lambda_j}{(\lambda_j+\tilde{\lambda}_n)^2}
\le
\frac{\lambda_{\max}(S_n)}{(\lambda_{\min}(S_n)+\tilde{\lambda}_n)^2},
\]
and hence
\[
\lambda_{\max}(AS_nA)\le \frac{\lambda_{\max}(S_n)}{(\lambda_{\min}(S_n)+\tilde{\lambda}_n)^2}.
\]
Combining the last two displays yields
\[
\tilde{\lambda}_n^2\,\operatorname{tr}\!\Big(B^\top AS_nAB\Big)
\le
\frac{\tilde{\lambda}_n^2\,\lambda_{\max}(S_n)}{(\lambda_{\min}(S_n)+\tilde{\lambda}_n)^2}\,\|B\|_F^2.
\]

For the variance bound, using the same diagonalization,
\[
\operatorname{tr}\!\Big(S_nAS_nA\Big)
=
\sum_{j=1}^p \frac{\lambda_j^2}{(\lambda_j+\tilde{\lambda}_n)^2}
\le
\sum_{j=1}^p \frac{\lambda_{\max}(S_n)^2}{(\lambda_{\min}(S_n)+\tilde{\lambda}_n)^2}
=
p\,\frac{\lambda_{\max}(S_n)^2}{(\lambda_{\min}(S_n)+\tilde{\lambda}_n)^2}.
\]
Multiplying by $\mathrm{tr}(\Sigma_\varepsilon)/n$ gives the desired variance bound.

\subsection{Proof of Lemma \ref{lem:gd_psd_general}}

Since $S_n\theta\in\operatorname{range}(S_n)$ for all $\theta$ and
$b\in\operatorname{range}(S_n)$, it follows that $\nabla g_n(\theta)\in\operatorname{range}(S_n)$
for all $\theta$. The gradient descent recursion with step size $\eta=1/L_n$,
\[
\theta^{(t+1)}=\theta^{(t)}-\frac1{L_n}\big(S_n\theta^{(t)}-b\big),
\qquad L_n:=\lambda_{\max}(S_n),
\]
therefore implies $\theta^{(t+1)}-\theta^{(t)}\in\operatorname{range}(S_n)$. Hence, if
$\theta^{(0)}\in\operatorname{range}(S_n)$, then $\theta^{(t)}\in\operatorname{range}(S_n)$
for all $t\ge0$. In particular, $e^{(t)}:=\theta^{(t)}-\theta^\star\in\operatorname{range}(S_n)$
for all $t\ge0$. Now using, $S_n\theta^\star=b$, we obtain
\begin{align}\label{eq_gd_linear1}
e^{(t+1)}
=\theta^{(t+1)}-\theta^\star
=\theta^{(t)}-\theta^\star-\frac1{L_n}\big(S_n\theta^{(t)}-S_n\theta^\star\big)
=\Big(I-\frac1{L_n}S_n\Big)e^{(t)}.
\end{align}
Moreover, expanding $g_n(\theta^\star+e)$ and using $b=S_n\theta^\star$ yields the exact identity
\[
g_n(\theta^\star+e)-g_n(\theta^\star)=\frac12\,e^\top S_n e,
\]
and hence
\[
g_n(\theta^{(t)})-g_n(\theta^\star)=\frac12\,(e^{(t)})^\top S_n e^{(t)}.
\]

Now if, $\{(\lambda_i,v_i)\}_{i=1}^r$ be an orthonormal eigenbasis of $\operatorname{range}(S_n)$,
so that $S_nv_i=\lambda_i v_i$ with $\mu_n \le \lambda_i\le L_n,
\quad i=1,\dots,r.$

Write $e^{(t)}=\sum_{i=1}^r \alpha_i^{(t)} v_i$. Then
\[
\alpha_i^{(t+1)}=\Big(1-\frac{\lambda_i}{L_n}\Big)\alpha_i^{(t)}.
\]
Therefore,
\[
\begin{aligned}
g_n(\theta^{(t+1)})-g_n(\theta^\star)
&=\frac12\sum_{i=1}^r \lambda_i\Big(\alpha_i^{(t+1)}\Big)^2
=\frac12\sum_{i=1}^r \lambda_i\Big(1-\frac{\lambda_i}{L_n}\Big)^2\Big(\alpha_i^{(t)}\Big)^2 \\
&\le \Big(1-\frac{\mu_n}{L_n}\Big)^2\cdot \frac12\sum_{i=1}^r \lambda_i\Big(\alpha_i^{(t)}\Big)^2 \\
&=\Big(1-\frac{\mu_n}{L_n}\Big)^2\big\{g_n(\theta^{(t)})-g_n(\theta^\star)\big\}.
\end{aligned}
\]
Iterating this inequality yields, for all $t\ge0$,
\[
g_n(\theta^{(t)})-g_n(\theta^\star)
\le \Big(1-\frac{\mu_n}{L_n}\Big)^{2t}\big\{g_n(\theta^{(0)})-g_n(\theta^\star)\big\}.
\]
Since $0<1-\mu_n/L_n<1$, we also have
$\big(1-\mu_n/L_n\big)^{2t}\le \big(1-\mu_n/L_n\big)^{t}$, which gives the second bound.
Finally, setting
\[
q_n:=1-\frac{\mu_n}{L_n}\in(0,1),
\]
the above displays can be written as $q_n^{2t}$ and $q_n^{t}$.

\acks 
We thank Roman Vershynin and Mark Rudelson for providing useful resources that helped enhance the quality of this paper.

\fund 
The first author's work was supported by NSF-DMS-2506060, the second author's work was supported by NSF-DMS-2506059, and the third author's work was supported by NSF-DMS-2506058.

\competing 
There were no competing interests to declare which arose during the preparation or publication process of this article.

%
%
%

\bibliography{ref}

@article{Dongarra,
author = {Chen, Zizhong and Dongarra, Jack J.},
title = {Condition Numbers of Gaussian Random Matrices},
journal = {SIAM Journal on Matrix Analysis and Applications},
volume = {27},
number = {3},
pages = {603-620},
year = {2005},
doi = {10.1137/040616413},
URL = { https://doi.org/10.1137/040616413},
eprint = { https://doi.org/10.1137/040616413}
}

@article{edelman,
author = {Edelman, Alan and Sutton, Brian D.},
title = {Tails of Condition Number Distributions},
journal = {SIAM Journal on Matrix Analysis and Applications},
volume = {27},
number = {2},
pages = {547-560},
year = {2005},
doi = {10.1137/040614256},
URL = {  https://doi.org/10.1137/040614256},
eprint = { https://doi.org/10.1137/040614256}
}

@misc{poggio2020doubledescentconditionnumber,
      title={Double descent in the condition number}, 
      author={Tomaso Poggio and Gil Kur and Andrzej Banburski},
      year={2020},
      eprint={1912.06190},
      archivePrefix={arXiv},
      primaryClass={cs.LG},
      url={https://arxiv.org/abs/1912.06190}, 
}

@misc{pan2012conditionnumbersrandomtoeplitz,
      title={Condition Numbers of Random Toeplitz and Circulant Matrices}, 
      author={Victor Y. Pan and Guoliang Qian},
      year={2012},
      eprint={1212.4551},
      archivePrefix={arXiv},
      primaryClass={math.NA},
      url={https://arxiv.org/abs/1212.4551}, 
}

@article{PhysRevE.90.050103,
  title = {Phase transitions in the condition-number distribution of Gaussian random matrices},
  author = {P\'erez Castillo, Isaac and Katzav, Eytan and Vivo, Pierpaolo},
  journal = {Phys. Rev. E},
  volume = {90},
  issue = {5},
  pages = {050103},
  numpages = {5},
  year = {2014},
  month = {Nov},
  publisher = {American Physical Society},
  doi = {10.1103/PhysRevE.90.050103},
  url = {https://link.aps.org/doi/10.1103/PhysRevE.90.050103}
}

@misc{manriquemirón2023conditionnumberrandomtridiagonal,
      title={Condition Number of Random Tridiagonal Toeplitz Matrix}, 
      author={Paulo Manrique-Mirón},
      year={2023},
      eprint={2305.11971},
      archivePrefix={arXiv},
      primaryClass={math.PR},
      url={https://arxiv.org/abs/2305.11971}, 
}

@article{Sankar,
author = {Sankar, Arvind and Spielman, Daniel A. and Teng, Shang-Hua},
title = {Smoothed Analysis of the Condition Numbers and Growth Factors of Matrices},
journal = {SIAM Journal on Matrix Analysis and Applications},
volume = {28},
number = {2},
pages = {446-476},
year = {2006},
doi = {10.1137/S0895479803436202},

URL = { 
    
        https://doi.org/10.1137/S0895479803436202
    
    

},
eprint = { 
    
        https://doi.org/10.1137/S0895479803436202
    
    

}
}

@misc{vershynin2011introductionnonasymptoticanalysisrandom,
      title={Introduction to the non-asymptotic analysis of random matrices}, 
      author={Roman Vershynin},
      year={2011},
      eprint={1011.3027},
      archivePrefix={arXiv},
      primaryClass={math.PR},
      url={https://arxiv.org/abs/1011.3027}, 
}

@article{Batir2017GammaBounds,
  author  = {Bat{\i}r, Necdet},
  title   = {Bounds for the Gamma Function},
  journal = {Results in Mathematics},
  year    = {2017},
  volume  = {72},
  pages   = {865--874},
  doi     = {10.1007/s00025-017-0698-0}
}

@article {ridge2,
    AUTHOR = {Dobriban, Edgar and Wager, Stefan},
     TITLE = {High-dimensional asymptotics of prediction: ridge regression
              and classification},
   JOURNAL = {Ann. Statist.},
  FJOURNAL = {The Annals of Statistics},
    VOLUME = {46},
      YEAR = {2018},
    NUMBER = {1},
     PAGES = {247--279},
      ISSN = {0090-5364,2168-8966},
   MRCLASS = {62H99 (62H30 62J07)},
  MRNUMBER = {3766952},
       DOI = {10.1214/17-AOS1549},
       URL = {https://doi.org/10.1214/17-AOS1549},
}

@article{ridge1,
  title={On the asymptotic risk of ridge regression with many predictors},
  author={Balasubramanian, Krishnakumar and Burman, Prabir and Paul, Debashis},
  journal={Indian Journal of Pure and Applied Mathematics},
  volume={55},
  number={3},
  pages={1040--1054},
  year={2024},
  publisher={Springer}
}

@article {gram,
    AUTHOR = {Koltchinskii, Vladimir and Lounici, Karim},
     TITLE = {Concentration inequalities and moment bounds for sample
              covariance operators},
   JOURNAL = {Bernoulli},
  FJOURNAL = {Bernoulli. Official Journal of the Bernoulli Society for
              Mathematical Statistics and Probability},
    VOLUME = {23},
      YEAR = {2017},
    NUMBER = {1},
     PAGES = {110--133},
      ISSN = {1350-7265,1573-9759},
   MRCLASS = {60E15 (60B11 60G15 62G99)},
  MRNUMBER = {3556768},
       DOI = {10.3150/15-BEJ730},
       URL = {https://doi.org/10.3150/15-BEJ730},
}

@book{vershynin2018high,
  author    = {Roman Vershynin},
  title     = {High-Dimensional Probability: An Introduction with Applications in Data Science},
  series    = {Cambridge Series in Statistical and Probabilistic Mathematics},
  publisher = {Cambridge University Press},
  address   = {Cambridge, UK},
  year      = {2018},
  isbn      = {9781108415194},
  doi       = {10.1017/9781108231596},
}

@book{wainwright2019high,
  author    = {Martin J. Wainwright},
  title     = {High-Dimensional Statistics: A Non-Asymptotic Viewpoint},
  series    = {Cambridge Series in Statistical and Probabilistic Mathematics},
  volume    = {48},
  publisher = {Cambridge University Press},
  address   = {Cambridge, UK},
  year      = {2019},
  isbn      = {9781108498029},
  doi       = {10.1017/9781108627771},
  note      = {A non-asymptotic viewpoint},
}

@article{RudelsonVershynin2009SmallestSingularRectangular,
  author  = {Rudelson, Mark and Vershynin, Roman},
  title   = {The smallest singular value of a random rectangular matrix},
  journal = {Comm. Pure Appl. Math.},
  volume  = {62},
  pages   = {1707--1739},
  year    = {2009},
  doi     = {10.1002/cpa.20294}
}

@article{Tikhomirov2016SmallestSingularNoMoments,
  author  = {Tikhomirov, Konstantin E.},
  title   = {The smallest singular value of random rectangular matrices with no moment assumptions on entries},
  journal = {Isr. J. Math.},
  volume  = {212},
  pages   = {289--314},
  year    = {2016},
  doi     = {10.1007/s11856-016-1287-8}
}

@article{Li1999MetricEntropySmallBall,
  author  = {Li, Wenbo V.},
  title   = {Approximation, metric entropy and small ball estimates for Gaussian measures},
  journal = {Ann. Probab.},
  volume  = {27},
  number  = {3},
  year    = {1999},
  doi     = {10.1214/aop/1022677459}
}

@article{Peng2009PartialCorrJointSparse,
  author  = {Peng, Jie and Wang, Pei and Zhou, Nengfeng and Zhu, Ji},
  title   = {Partial Correlation Estimation by Joint Sparse Regression Models},
  journal = {J. Am. Stat. Assoc.},
  year    = {2009},
  volume  = {104},
  number  = {486},
  pages   = {735--746},
  month   = jun,
  doi     = {10.1198/jasa.2009.0126},
  pmid    = {19881892},
  pmcid   = {PMC2770199}
}

@book{Muirhead1982Aspects,
  title     = {Aspects of Multivariate Statistical Theory},
  author    = {Muirhead, Robb J.},
  publisher = {John Wiley \& Sons},
  year      = {1982},
  address   = {New York},
  series    = {Wiley Series in Probability and Mathematical Statistics}
}

@book{Anderson2003Multivariate,
  title     = {An Introduction to Multivariate Statistical Analysis},
  author    = {Anderson, Theodore W.},
  edition   = {3},
  publisher = {Wiley},
  year      = {2003},
  address   = {Hoboken, NJ}
}

@article{LedoitWolf2004WellConditioned,
  title   = {A Well-Conditioned Estimator for Large-Dimensional Covariance Matrices},
  author  = {Ledoit, Olivier and Wolf, Michael},
  journal = {Journal of Multivariate Analysis},
  volume  = {88},
  number  = {2},
  pages   = {365--411},
  year    = {2004},
  doi     = {10.1016/S0047-259X(03)00096-4}
}

@article{bubeck2015convex, title={Convex Optimization: Algorithms and Complexity}, author={Bubeck, S{\'e}bastien}, journal={Foundations and Trends{\textregistered} in Machine Learning}, volume={8}, number={3--4}, pages={231--357}, year={2015}, publisher={Now Publishers} }

@book{saad2003iterative,
  title        = {Iterative Methods for Sparse Linear Systems},
  author       = {Saad, Yousef},
  edition      = {2},
  publisher    = {Society for Industrial and Applied Mathematics (SIAM)},
  address      = {Philadelphia, PA},
  year         = {2003},
  doi          = {10.1137/1.9780898718003}
}

@book{trefethen1997numerical,
  title        = {Numerical Linear Algebra},
  author       = {Trefethen, Lloyd N. and Bau, David},
  publisher    = {Society for Industrial and Applied Mathematics (SIAM)},
  address      = {Philadelphia, PA},
  year         = {1997},
  isbn         = {978-0-89871-361-9}
}

@article {bickel2008regularized,
    AUTHOR = {Bickel, Peter J. and Levina, Elizaveta},
     TITLE = {Regularized estimation of large covariance matrices},
   JOURNAL = {Ann. Statist.},
  FJOURNAL = {The Annals of Statistics},
    VOLUME = {36},
      YEAR = {2008},
    NUMBER = {1},
     PAGES = {199--227},
      ISSN = {0090-5364},
   MRCLASS = {62H12 (62F12 62G09)},
  MRNUMBER = {2387969},
       DOI = {10.1214/009053607000000758},
       URL = {https://doi.org/10.1214/009053607000000758},
}

@article {xiang,
    AUTHOR = {Xiang, Ruoxuan and Khare, Kshitij and Ghosh, Malay},
     TITLE = {High dimensional posterior convergence rates for decomposable
              graphical models},
   JOURNAL = {Electron. J. Stat.},
  FJOURNAL = {Electronic Journal of Statistics},
    VOLUME = {9},
      YEAR = {2015},
    NUMBER = {2},
     PAGES = {2828--2854},
   MRCLASS = {62H99 (62F15 62G20)},
  MRNUMBER = {3439186},
       DOI = {10.1214/15-EJS1084},
       URL = {https://doi.org/10.1214/15-EJS1084},
}

@article {spectrum,
    AUTHOR = {El Karoui, Noureddine},
     TITLE = {Spectrum estimation for large dimensional covariance matrices
              using random matrix theory},
   JOURNAL = {Ann. Statist.},
  FJOURNAL = {The Annals of Statistics},
    VOLUME = {36},
      YEAR = {2008},
    NUMBER = {6},
     PAGES = {2757--2790},
      ISSN = {0090-5364},
   MRCLASS = {62H12 (62-09)},
  MRNUMBER = {2485012},
MRREVIEWER = {Jussi S. Klemel\"{a}},
       DOI = {10.1214/07-AOS581},
       URL = {https://doi.org/10.1214/07-AOS581},
}

@article {sarkar,
    AUTHOR = {Sarkar, Partha and Khare, Kshitij and Ghosh, Malay},
     TITLE = {Posterior consistency in multi-response regression models with
              non-informative priors for the error covariance matrix in
              growing dimensions},
   JOURNAL = {Bernoulli},
  FJOURNAL = {Bernoulli. Official Journal of the Bernoulli Society for
              Mathematical Statistics and Probability},
    VOLUME = {31},
      YEAR = {2025},
    NUMBER = {3},
     PAGES = {2403--2433},
      ISSN = {1350-7265,1573-9759},
   MRCLASS = {62J05 (62F12 62F15 62H12)},
  MRNUMBER = {4889264},
MRREVIEWER = {Weiping\ Zhang},
       DOI = {10.3150/24-bej1810},
       URL = {https://doi.org/10.3150/24-bej1810},
}

@article {bai_yin,
    AUTHOR = {Bai, Z. D. and Yin, Y. Q.},
     TITLE = {Limit of the smallest eigenvalue of a large-dimensional sample
              covariance matrix},
   JOURNAL = {Ann. Probab.},
  FJOURNAL = {The Annals of Probability},
    VOLUME = {21},
      YEAR = {1993},
    NUMBER = {3},
     PAGES = {1275--1294},
      ISSN = {0091-1798,2168-894X},
   MRCLASS = {60F15 (62H99)},
  MRNUMBER = {1235416},
MRREVIEWER = {Jack\ W.\ Silverstein},
       URL =
              {http://links.jstor.org/sici?sici=0091-1798(199307)21:3<1275:LOTSEO>2.0.CO;2-2&origin=MSN},
}

@article{Penrose_1955,
  title = {A Generalized Inverse for Matrices},
  author = {Penrose, R.},
  journal = {Math. Proc. Camb. Philos. Soc.},
  volume = {51},
  number = {3},
  pages = {406--413},
  year = {1955}
}

@article {bai_ghosh,
    AUTHOR = {Bai, Ray and Ghosh, Malay},
     TITLE = {High-dimensional multivariate posterior consistency under
              global-local shrinkage priors},
   JOURNAL = {J. Multivariate Anal.},
  FJOURNAL = {Journal of Multivariate Analysis},
    VOLUME = {167},
      YEAR = {2018},
     PAGES = {157--170},
      ISSN = {0047-259X},
   MRCLASS = {62F12 (62F15 62H12 62J07)},
  MRNUMBER = {3830639},
       DOI = {10.1016/j.jmva.2018.04.010},
       URL = {https://doi.org/10.1016/j.jmva.2018.04.010},
}

@article {armagan,
    AUTHOR = {Armagan, A. and Dunson, D. B. and Lee, J. and Bajwa, W. U. and
              Strawn, N.},
     TITLE = {Posterior consistency in linear models under shrinkage priors},
   JOURNAL = {Biometrika},
  FJOURNAL = {Biometrika},
    VOLUME = {100},
      YEAR = {2013},
    NUMBER = {4},
     PAGES = {1011--1018},
      ISSN = {0006-3444},
   MRCLASS = {62F15 (62J05 62J07)},
  MRNUMBER = {3142348},
       DOI = {10.1093/biomet/ast028},
       URL = {https://doi.org/10.1093/biomet/ast028},
}

@article {littlewood_offword,
    AUTHOR = {Rudelson, Mark and Vershynin, Roman},
     TITLE = {The {L}ittlewood-{O}fford problem and invertibility of random
              matrices},
   JOURNAL = {Adv. Math.},
  FJOURNAL = {Advances in Mathematics},
    VOLUME = {218},
      YEAR = {2008},
    NUMBER = {2},
     PAGES = {600--633},
      ISSN = {0001-8708,1090-2082},
   MRCLASS = {60E15 (60B20)},
  MRNUMBER = {2407948},
MRREVIEWER = {Ben\ Joseph\ Green},
       DOI = {10.1016/j.aim.2008.01.010},
       URL = {https://doi.org/10.1016/j.aim.2008.01.010},
}

@book{alan_edelman,
    AUTHOR = {Edelman, Alan Stuart},
     TITLE = {Eigenvalues and condition numbers of random matrices},
      NOTE = {Thesis (Ph.D.)--Massachusetts Institute of Technology},
 PUBLISHER = {ProQuest LLC, Ann Arbor, MI},
      YEAR = {1989},
     PAGES = {(no paging)},
   MRCLASS = {Thesis},
  MRNUMBER = {2941174},
       URL = {http://gateway.proquest.com/openurl?url_ver=Z39.88-2004&rft_val_fmt=info:ofi/fmt:kev:mtx:dissertation&res_dat=xri:pqdiss&rft_dat=xri:pqdiss:0379924}
}

@article {Hastie,
    AUTHOR = {Hastie, Trevor and Montanari, Andrea and Rosset, Saharon and
              Tibshirani, Ryan J.},
     TITLE = {Surprises in high-dimensional ridgeless least squares
              interpolation},
   JOURNAL = {Ann. Statist.},
  FJOURNAL = {The Annals of Statistics},
    VOLUME = {50},
      YEAR = {2022},
    NUMBER = {2},
     PAGES = {949--986},
      ISSN = {0090-5364,2168-8966},
   MRCLASS = {62J05 (62F12 62J02 62J07)},
  MRNUMBER = {4404925},
       DOI = {10.1214/21-aos2133},
       URL = {https://doi.org/10.1214/21-aos2133},
}
\bibliographystyle{APT}

\end{document}